\documentclass[11pt,reqno,a4paper]{amsart}

\pdfoutput=1

\usepackage{amsmath,amssymb,amsthm, amsrefs}
\usepackage{mathtools}
\usepackage[margin=1in]{geometry}
\usepackage{enumerate}
\usepackage{tikz}
\usepackage{iftex}
\usepackage{yhmath}
\usepackage{cancel}
\usepackage{color}
\usepackage{siunitx}
\usepackage{array}
\usepackage{tabularx}
\usepackage{tikz-cd}
\usepackage{graphicx}
\usepackage{hyperref}
\hypersetup{
	hidelinks 
}
\allowdisplaybreaks

\newtheorem{thm}{Theorem}[section]
\newtheorem{cor}[thm]{Corollary}
\newtheorem{lem}[thm]{Lemma}
\newtheorem{prop}[thm]{Proposition}

\newtheorem{ques}[thm]{Question}

\theoremstyle{definition}
\newtheorem{defn}[thm]{Definition}
\newtheorem{rem}[thm]{Remark}
\newtheorem{ex}[thm]{Example}
\newtheorem*{ack}{Acknowledgments}
\numberwithin{equation}{section}
\allowdisplaybreaks[1]

\def\bR{\mathbb{R}}
\def\bC{\mathbb{C}}
\def\bCP{\mathbb{CP}}
\def\bRP{\mathbb{RP}}
\def\bZ{\mathbb{Z}}

\title[Lagrangian mean curvature flow in the complex projective plane]{Lagrangian mean curvature flow in the\\complex projective plane}
\author{Christopher G. Evans}

\begin{document}

	\begin{abstract}
		We prove a Thomas--Yau-type conjecture for monotone Lagrangian tori satisfying a symmetry condition in the complex projective plane $\bCP^2$. We show that such tori exist for all time under Lagrangian mean curvature flow with surgery, undergoing at most a finite number of surgeries before flowing to a minimal Clifford torus in infinite time. Furthermore, we show that we can construct a torus with any finite number of surgeries before convergence. Along the way, we prove many interesting subsidiary results and develop methods which should be useful in studying Lagrangian mean curvature flow in non-Calabi--Yau manifolds, even in non-symmetric cases.
	\end{abstract}

	\maketitle

	\section{Introduction}
	
	Starting from the Clifford torus, Vianna (\cite{vianna_2014},\cite{Vianna2016}) constructs by an iterative sequence of mutations an infinite family $\mathcal F$ of monotone Lagrangian tori in the complex projective plane $\bCP^2$, with no two members of $\mathcal F$ Hamiltonian isotopic. In this paper, we make and explore a simple observation: Vianna's mutation is exactly the surgery procedure one expects at the prototypical Lawlor neck-type singularity in Lagrangian mean curvature flow. By restricting to equivariant examples of $\mathcal F$, those satisfying an $(S^1 \times \bZ_2)$-symmetry to be specified later, we define a surgery called neck-to-neck surgery at all possible singularities and are able to prove a Thomas--Yau-type conjecture (\cite{Thomas2002},\cite{joyce_2015}):
	
	\begin{thm}[Main theorem]\label{thm-main-sum}
		Let $F:L \to \bCP^2$ be an embedded equivariant monotone Lagrangian torus in the complex projective plane with the standard Fubini--Study K\"ahler metric. Then under Lagrangian mean curvature flow with surgery, $L$ exists for all time and, after undergoing at most a finite number of surgeries, converges to a minimal Clifford torus in infinite time. 
	\end{thm}
	
	This is a natural extension of the classical result that an embedded circle in $\bCP^1 = S^2$ that divides $S^2$ into equal area pieces exists for all time and converges to an equator. Unlike in the curve-shortening case, singularities are an inevitable feature of Lagrangian mean curvature flow in higher dimensions, and therefore developing an understanding of surgeries at these singularities is of crucial importance. We highlight that to the author's knowledge, this is the first example of Lagrangian mean curvature flow with surgery in the literature.
	
	In the course of proving the main theorem, we prove some interesting side results, a selection of which we now highlight. 
	
	Wang \cite{Wang2001} proved that almost-calibrated Lagrangians in Calabi--Yau manifolds do not attain type I singularities. This result provides a clear distinction between Lagrangian mean curvature flow and hypersurface mean curvature flow where type I singularities are commonplace and type II singularities are rare. We prove an analogue of this result for monotone Lagrangians in Fano manifolds, which we take to mean K\"ahler--Einstein manifolds with positive Einstein constant $\kappa$:
	
	\begin{thm}\label{Monotone-typeI-sum}
		Let $F_t:L^n \to M^{2n}$ be a monotone Lagrangian mean curvature flow in a K\"ahler--Einstein manifold $M$ with Einstein constant $\kappa > 0$. Then $F_t$ does not attain any type I singularities.
	\end{thm}
	
	By studying the equation for the equivariant mean curvature, which we derive quickly and explicitly by a novel method, we are able to show that:
	
	\begin{thm}\label{thm-main-minimal}
		There exists a countably infinite family of complete immersed minimal  Lagrangian tori in $\bCP^2$.
	\end{thm}
	
	The immersed minimal tori constructed are similar to the self-shrinking solutions to curve-shortening flow discovered by Abresch--Langer \cite{abresch_langer_1986}, see Figure \ref{fig-spirgoraphs}.
		
	\begin{figure}[b]
		\centering
		\includegraphics[scale=0.24]{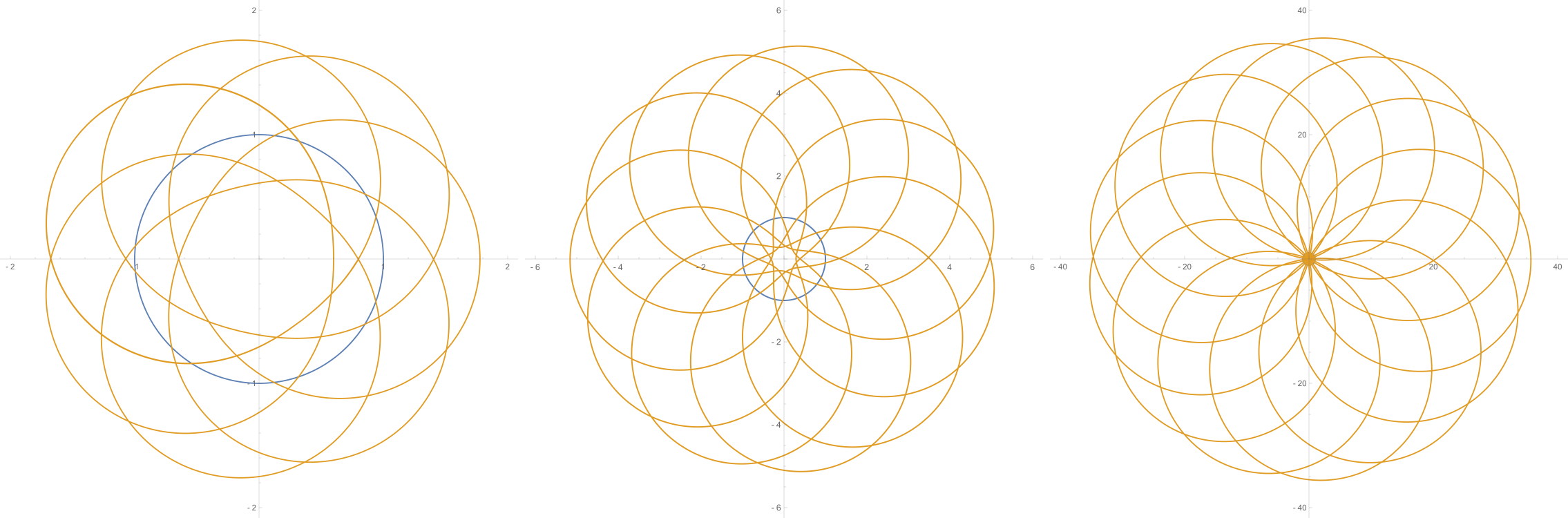}
		\caption{The profile curves of a selection of complete immersed minimal equivariant Lagrangian tori in $\bCP^2$ of varying period.}\label{fig-spirgoraphs}
	\end{figure}
	
	The equivariant restriction in the main theorem leaves us with two classes Lagrangian tori: up to Hamiltonian isotopy, they are the Clifford torus and the Chekanov torus. The main theorem above might seem to imply that a Clifford torus does not have singularities under Lagrangian mean curvature flow. However, this is not the case:
	
	\begin{thm}\label{thm-Clifford-finite-time-sing-sum}
		There exists a Clifford torus $L$ in $\bCP^2$ such that under mean curvature flow $L$ has a finite-time singularity and surgery at the singularity makes $L$ a Chekanov torus. 
	\end{thm}
	
	By iterating the explicit construction we provide in a relatively straightforward manner, it is also clear that for any $n \geq 0$, one can find a Lagrangian torus $L$ in $\bCP^2$ with exactly $n$ singularities before converging to a minimal Clifford torus.
	
	Thus the behaviour we observe is of a cyclical nature: a Clifford torus can collapse to become a Chekanov torus, which then exists for some time after surgery before collapsing to a Clifford torus again. This process can repeat any finite number of times before eventually becoming a stable Clifford torus. We note that this type of ``flip-flopping'' behaviour is unusual in mean curvature flow, and to the best of the author's knowledge has not been observed before. We conjecture that it is in fact not unusual behaviour for Lagrangian mean curvature flow, and that we should expect to see similar behaviour in Calabi--Yau manifolds as well.
	
	Despite the restrictive nature of the symmetry used, we believe that the ideas and principles used to guide the proofs here have the potential to see further application in the non-symmetric case. The primary reason for the symmetry is not for symplectic or topological reasons, but to reduce the type of singularities that can occur to a case which is well-understood by the work of Wood (\cite{Wood2019},\cite{wood_2020}). Our main contribution on this front is to define a surgery using the Scale Lemma (Lemma \ref{lem-scale}), showing that singularity formation happens on an arbitrarily small scale. This allows us to categorise the behaviour in the proof of Theorem \ref{thm-main-sum} by observing that the number of intersections with a specific pair of real projective planes decreases under the flow with surgery. Fascinatingly, no such result exists in the Calabi--Yau case and the method of proof employed here cannot possibly be generalised. 
	
	Let us discuss the proof of the main theorem, in the process giving a guide to the paper. 
	
	In Section \ref{sec-equiv}, we begin by calculating the governing equation for equivariant mean curvature flow in $\bCP^2$. Here, we use a novel method to avoid raw computation: we relate the mean curvature of a Lagrangian to the relative Lagrangian angle it forms with the standard toric fibration of $\bCP^2$ by Lagrangian Clifford tori. This vastly simplifies the calculations.
	
	In Section \ref{sec-triangle}, we introduce one of the key ideas which enables the majority of the rest of the paper. Using a mild generalisation of the Cieliebak--Goldstein theorem \cite{Cieliebak2004} to include Lagrangians with corners, we consider evolution equations of areas of $J$-holomorphic triangles bounded by segments of our flowing Lagrangian tori and Lagrangian cones given by unions of totally geodesic real projective planes in $\bRP^2$. The key insight is that by considering such triangles between Hamiltonian non-isotopic objects, we are considering objects which are in a Floer-theoretic sense non-trivial. Thus the behaviour of these triangles is geometric rather than topological, and hence can be measured with respect to the mean curvature flow.
		
	The main mathematical debt of this paper is owed to the work of Neves (\cite{Neves2007},\cite{Neves2010}) and the work of Wood (\cite{Wood2019},\cite{wood_2020}) on singularity formation in Lagrangian mean curvature flow. In Section \ref{sec-sing}, we replicate their results in the positive curvature setting. In doing so, we restrict the class of singularities that can form to simply one type: Lawlor neck singularities at the origin with type I blow-up given by a specific cone of opening angle $\pi/2$, namely $C^0_{\pi/2}$. A full description of the singular behaviour is given in Section \ref{sec-sing}, although since the proof of this fact is long and rather tedious, we refer the reader to the author's doctoral thesis for a full proof. In addition, we prove the Scale Lemma using properties of the minimal equivariant Lagrangians constructed in Section \ref{sec-minimal}.
	
	Next, we study equivariant Clifford and Chekanov tori. We show in Section \ref{sec-graph-Cliff} that Clifford tori satisfying a graphical condition have long-time existence and convergence to the minimal equivariant Clifford torus. We then show in Section \ref{sec-Chek} that any Chekanov torus has a finite-time singularity under mean curvature flow. In the process, we show that there is no minimal equivariant Chekanov torus. The method of proof, using a triangle calculation as described above, is tantalisingly close to being generalisable to non-equivariant tori. 
	
	In Section \ref{sec-surgery}, we define a neck-to-neck surgery using the Scale Lemma that strictly decreases the number of intersections with the cone $C^0_{\pi/2}$. This allows us to handle the remaining cases, and prove the main theorem. 
	
	\begin{ack}
		This paper is a reorganised and abridged version of the author's doctoral thesis \cite{Evans_2022}.
		
		I would like to thank my supervisors Jason Lotay and Felix Schulze for their constant support and encouragement. Many thanks are owed to Jonny Evans for introducing me to the work of Vianna, and encouraging me to study it in the context of mean curvature flow. I would also like to thank my thesis examiners for their many useful comments and corrections.
		
		Thank you to Emily Maw for patiently explaining almost-toric fibrations to me. Thank you to Ben Lambert for many helpful talks over the years, and for struggling through some particularly heinous calculations to verify one of my results. Finally, thank you to Albert Wood, who not only explained in careful detail a great number of his results to me, but also read preliminary versions of my thesis and provided important feedback.
		
		This work was supported by the Leverhulme Trust Research Project Grant RPG-2016-174.
	\end{ack}

	\section{Preliminaries}
	\subsection{Lagrangian tori in \texorpdfstring{$\bCP^2$}{CP2}}
	
	Let $(M,g,J,\omega)$ be a K\"ahler--Einstein manifold with Einstein constant $\kappa$. We call $M$ a Fano manifold if $\kappa>0$. This definition is non-standard in the literature, but is most appropriate from the point of view of Lagrangian mean curvature flow. A half-dimensional immersed submanifold $F:L^n \to M^{2n}$ is called Lagrangian if the symplectic form $\omega$ vanishes on $L$, i.e. $F^*\omega = 0$. We will frequently refer to the immersed Lagrangian $F(L)$ by $L$ where it is unlikely to cause confusion. 
	
	Lagrangians contain a great deal of information about the symplectic topology of the manifold they live in, but since symplectic geometry is always trivial locally, it is often necessary to restrict to a subclass of Lagrangians in order to obtain interesting results. For Fano manifolds, the most important subclass is that of monotone Lagrangians. An embedded Lagrangian $L$ is called monotone if for any disc $D \in \pi_2(M,L)$, the Maslov class of $D$ is related to its holomorphic area by 
	\[\kappa \int_D \omega =  \pi \mu(D).\]
	
	The prototypical and arguably most important Fano manifolds are the complex projective spaces $\bCP^n$, realised as quotients of Euclidean space $\bC^{n+1}$ by the Hopf fibration
	\[(z_0,\dots,z_n) \mapsto [z_0:\cdots:z_n],\]
	where $[z_0:\cdots:z_n] = [w_0:\cdots:w_n]$ if and only if $(w_0,\dots,w_n) = \lambda (z_0,\dots,z_n)$, for some $\lambda \in \bC$. Of course, the Hopf fibration can also be taken from the unit $S^{2n+1}$ sphere in $\bC^{n+1}$. 
	
	The first complex projective space, the complex projective line $\bCP^1$ is the round 2-sphere $S^2$ with Einstein constant $\kappa = 4$. Since we only have 2 real dimensions, it is easy to find all the Lagrangians: any curve in $\bCP^1$ is Lagrangian since $F^*\omega$ is always zero. It is also easy to spot the monotone Lagrangians: an embedded circle $F:S^1 \to \bCP^1$ is monotone if and only if $F(S^1)$ divides the sphere into two pieces of equal area. All monotone Lagrangians are Hamiltonian isotopic.
	
	The second complex projective space, the complex projective plane $\bCP^2$, has Einstein constant $\kappa=6$, which immediately distinguishes it from the $4$-sphere $S^4$. From a symplectic point of view, $\bCP^2$ is vastly more complicated than $\bCP^1$, as evidenced by the abundance of interesting monotone Lagrangians one can construct. The first encountered monotone Lagrangian is the Clifford torus $L_{\text{Cl}}$ given by the projection of the  ``equator'' of the $5$-sphere $S^5 \subset \bC^3$:
	\[(e^{i\theta_0}, e^{i\theta_1}, e^{i \theta_2}) \mapsto [e^{i\theta_0}: e^{i\theta_1}: e^{i\theta_2}].\]
	Unlike $\bCP^1$, there are a great wealth of embedded monotone tori not Hamiltonian isotopic to $L_{\text{Cl}}$. The first is the Chekanov torus, discovered in the 90s \cite{schlenk_chekanov_2010}. One way to obtain the Chekanov torus is from the Clifford torus by a process known as a mutation, which we now discuss.
	
	\subsubsection{Vianna's exotic tori in \texorpdfstring{$\bCP^2$}{CP2}}
	
	Vianna constructs in \cite{Vianna2016} an infinite family $\mathcal F$ of monotone Lagrangian tori in $\bCP^2$, with no two tori Hamiltonian isotopic. We present some details of this construction here as it forms the main motivating example for the rest of the paper.
	
	The first member of the family $\mathcal F$ is the Clifford torus
	\[L_{\text{Cl}} = \{[x:y:z] : |x|=|y|=|z|=1\},\]
	which is realised as the barycentric fibre in the standard toric fibration of $\bCP^2$. From here, Vianna constructs the next member of $\mathcal F$ by a topological procedure known as a \textit{mutation}:
	
	\begin{enumerate}
		\item Introduce a nodal fibre at one of the corners by a \textit{nodal trade}. The corner of the base diagram is now a circle and the fibre above the cross is a Lagrangian torus pinched to create a nodal singularity. The barycentric fibre is still a Clifford torus, and the metric is still the Fubini--Study metric.
		
		\item Rescale a neighbourhood of the line at $\infty$ (i.e. the $\bCP^1$ given by $\{[x:y:0]:x,y \in \bC\}$) until the barycentre has passed over the nodal fibre. The barycentre is now a Chekanov torus and the metric is no longer the Fubini--Study metric.
		
		\item Isotope the metric back to the Fubini--Study metric using Moser's trick. The barycentre remains a Chekanov torus. 
	\end{enumerate} 
	
	Items 2 and 3 together are called a \textit{nodal slide}, and the full mutation is illustrated in Figure \ref{fig-mutation1}.
	
	\begin{figure}
		\centering
		\includegraphics[scale=1.5]{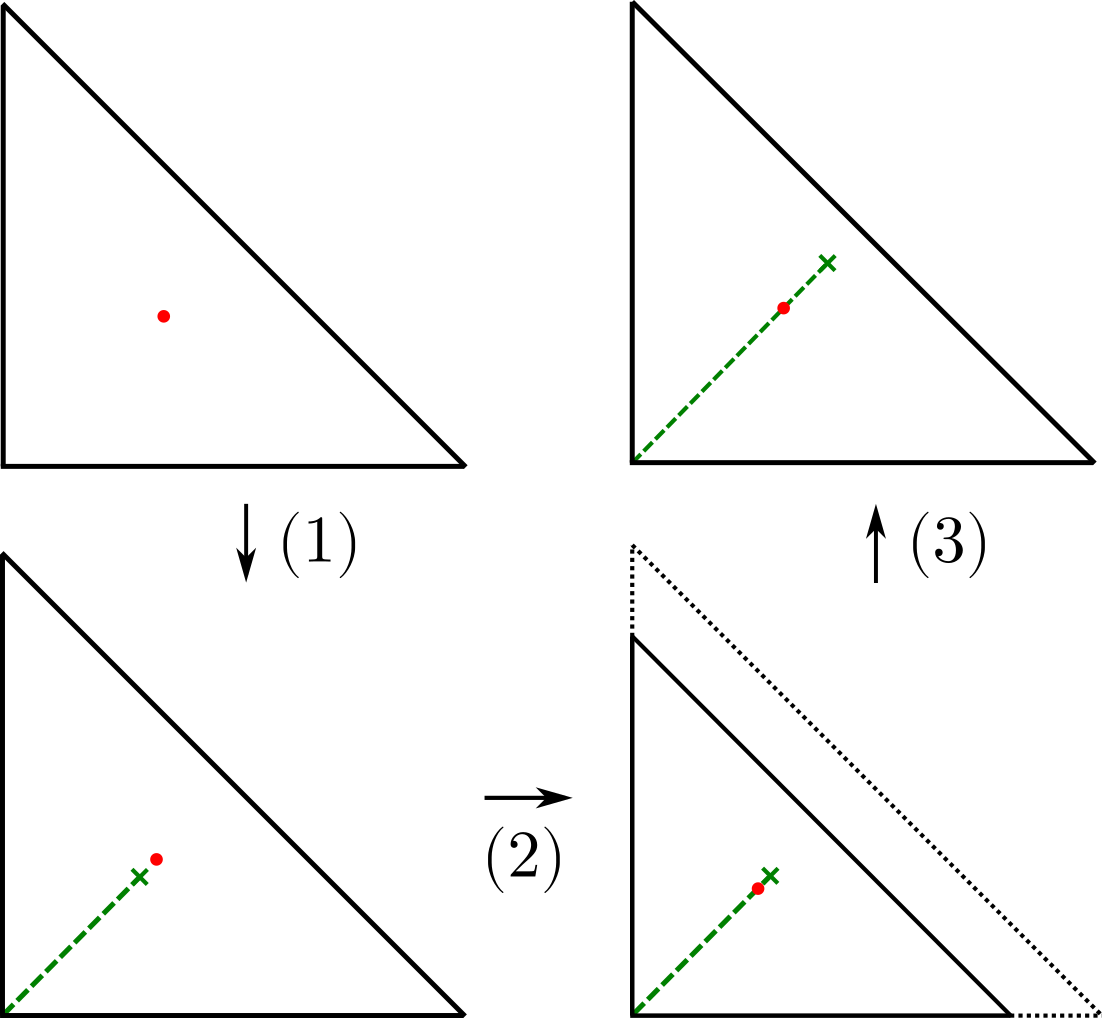}
		\caption{The mutation procedure. The fibre above the barycentres (red dots) are Clifford tori on the left-hand side, and Chekanov tori on the right-hand side. }\label{fig-mutation1}
	\end{figure}
	
	Vianna then iterates this procedure, introducing new nodal fibres at different corners of the moment polytope. Since the two corners $(1,0)$ and $(0,1)$ are the same after the first mutation, the Chekanov torus $L_{\text{Ch}}$ becomes a unique new torus $L_{(1,4,25)}$. Iterating further from this point generates two new tori every time. Vianna indexes this family by integer triples $(a^2,b^2,c^2)$ with $a^2+b^2+c^2 = 3abc$, which are known as \textit{Markov triples}, and shows that a torus $L_{(a^2,b^2,c^2)}$ is realised as the barycentric fibre of a degeneration of the weighted projective space $\bCP^2(a^2,b^2,c^2)$, though we shall not use this perspective in this paper.
	\begin{figure}\label{Fig-exotic-tori-tree}
		\centering
		\begin{tikzcd}[row sep=0.3em]
			& & & & \cdots \\
			& & &L_{(1,25,169)}  \arrow[rd, end anchor=  west] \arrow[ru, end anchor=  west] & \\ 
			& & & & \cdots \\
			L_{\text{Cl}} \ar[r] &L_{\text{Ch}} \ar[r]& L_{(1,4,25)} \arrow[rdd, end anchor = north west] \arrow[ruu,end anchor= south west] & &\\
			& & & & \cdots \\
			& & &L_{(4,25,841)} \arrow[rd, end anchor=  west] \arrow[ru, end anchor=  west] & \\
			& & & & \cdots
		\end{tikzcd}
		\caption{Vianna's exotic tori, indexed by Markov triples $(a^2,b^2,c^2)$.}
	\end{figure}
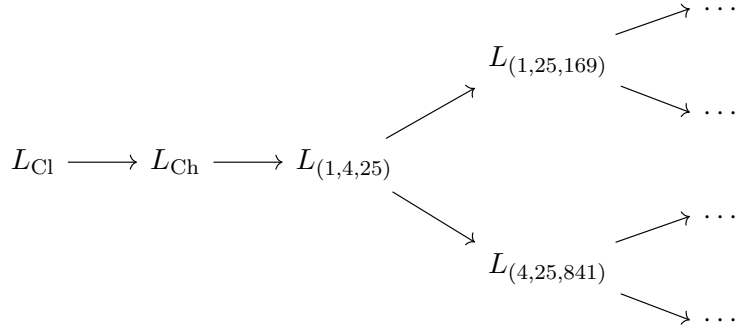
	We focus for the rest of the paper only on the first level of this procedure, which has been known since the work of Chekanov--Schlenk \cite{schlenk_chekanov_2010}. 
	
	We can distinguish tori of Clifford-type $L_{\text{Cl}}$ from tori of Chekanov-type $L_{\text{Ch}}$ by counting $J$-holomorphic disc classes in $H_2(\bCP^2,L)$. Following the results of Auroux \cite{Auroux_2007}, we have that there are 3 classes of Maslov 2 discs with boundary on $L_{\text{Cl}}$. Denote by $\alpha_1$ the disc class $w \mapsto [w:1:1]$, and by $\alpha_2$ the disc class $w \mapsto [1:w:1]$. Then $H_2(\bCP^2, L_{\text{Cl}})$ is generated by $\{\alpha_1,\alpha_2,Q-\alpha_1 -\alpha_2\}$ where $Q = [\bCP^1]$ is the hyperplane class. 
	
	On the other hand, consider the Chekanov torus in $\bC^2 = \bCP^2 -\{z=0\}$ given by 
	\[L_{\text{Ch}} = \left\{\left( \gamma(s) e^{i\alpha}, \gamma(s) e^{-i\alpha} \right) : s, \alpha \in \bR \right\},\]
	where $\gamma(s)\in \bC$ is an embedded closed curve not enclosing the origin. By an abuse of notation, let $\alpha$ denote the disc class of a disc with boundary given by the $\alpha$ coordinate, and by $\beta$ the disc class given by the $s$ coordinate. Then $H_2(\bCP^2,L_{\text{Ch}})$ is generated by $\{\alpha,\beta,Q\}$. However, the class $\alpha$ does not contain any holomorphic representatives - this is precisely the same reason the corresponding class can collapse for the Clifford torus $L_{\text{Cl}}$ under mean curvature flow in $\bC^2$ as demonstrated in \cite{Evans2018}. In fact, the Maslov 2 classes on $L_{\text{Ch}}$ are precisely
	\[\left\{\beta, Q-2\beta + \alpha, Q- 2\beta, Q-2\beta - \alpha\right\},\]
	each occurring with moduli space of holomorphic discs of dimension 1 except for the $Q-2\beta$ class which has dimension 2.
		
	\subsection{Lagrangian mean curvature flow}
	
	It is a fact first proven by Smoczyk \cite{Smoczyk1996} that the gradient-descent flow of area, i.e. mean curvature flow, preserves the Lagrangian condition in K\"ahler--Einstein manifolds. This gives rise to Lagrangian mean curvature flow. The preservation of the Lagrangian condition is no coincidence; for a Lagrangian in a K\"ahler--Einstein manifold, the mean curvature $\vec H$ can be written as a closed $1$-form $H$ on $L$, making use of the isomorphism between the tangent, normal and cotangent bundles of a Lagrangian given by the K\"ahler condition.
	
	Within the Lagrangian class, there are further preserved conditions. If the ambient manifold is Calabi--Yau, the most important preserved types are zero-Maslov and almost-calibrated, which have attracted a great deal of interest. For Fano manifolds, the monotone condition is the most important preserved quantity (we provide an elementary proof that the monotone condition is preserved in Section \ref{sec-CG-thm}.) 
	
	Singularities of the flow occur at times $T$ when $|A|^2 \to \infty$ as $t\to T$, where $A$ is the second fundamental form. As in hypersurface mean curvature flow, singularities for Lagrangian mean curvature flow can be divided into two types. A type I singularity occurs when the rate at which $|A|^2$ blows up is at most parabolic in time. All other singularities are of type II. As a general principle, the important preserved classes mentioned above do not have type I singularities; this statement is due to Wang \cite{Wang2001} in the Calabi--Yau case. We prove that monotone Lagrangians in Fano manifolds do not attain type I singularities in Section \ref{sec-monotone-type-I}.
	
	As an example, consider the case of Lagrangians in $\bCP^1 = S^2$. It is well-known that an embedded circle has a finite-time singularity if and only if it is not monotone. In this case, the singularity is of type I. On the other hand, monotone Lagrangians in $\bCP^1$, i.e. circles dividing the sphere into equal area pieces, exist for all time under mean curvature flow and converge to a geodesic equator. 
	
	Results of the latter type are known as Thomas--Yau-type results. The general idea is that the important preserved classes should exist for all time and converge to minimal Lagrangians. However, this is known to be false since for higher-dimensional examples, singularities are inevitable. Finite-time singularities are unavoidable, hence attention has turned to methods to resolve finite-time singularities via surgery in order to continue the flow. Geometric flows with surgery are well-studied, for instance the ground-breaking work of Perelman \cite{Perelman2002} on Ricci flow with surgery, or the work of Huisken--Sinistrarri \cite{Huisken2009} on mean convex mean curvature flow with surgery. However, without proper understanding of singularity formation, it is impossible to perform surgery. This has presented the most difficult obstacle to defining a Lagrangian mean curvature flow with surgery. 
	
	We note one final property of singularities of mean curvature flow. In general, geometric flows may have infinite-time singularities; indeed, this occurs in Ricci flow and in Yang--Mills flow, amongst others. In mean curvature flow however, one can rule out infinite-time singularities in certain cases. For simplicity, we state the following result of Chen--He \cite{chen_he_2010} as we need it in this paper, though it applies in a far wider generality.
	
	\begin{prop}\label{Prop-infinite-sing}
		Let $L$ be a Lagrangian mean curvature flow of a compact Lagrangian in a compact K\"ahler--Einstein manifold $M$ with $\kappa>0$.  Then $L$ either attains a finite-time singularity or has uniformly bounded $|A|^2$ for all time and converges subsequentially to a minimal Lagrangian submanifold in $M$ in infinite time.
	\end{prop}	
	
	\subsection{Holomorphic volume forms and the Lagrangian angle}
	
	Let $\Omega$ be a holomorphic volume form  defining a Lagrangian angle by $\Omega_L = e^{i\theta} \operatorname{vol}_L$. If $\Omega$ is parallel, as is the case in Calabi--Yau manifolds, we have the following important result:
	
	\begin{prop}\label{H = dtheta}
		Let $L$ be an oriented Lagrangian in a Calabi--Yau manifold $M$ with mean curvature 1-form $H$, where $H(X) = \omega(X, \vec H)$. Then $H=d \theta$, where $\theta$ is the Lagrangian angle.
	\end{prop}
	
	\begin{proof}
		Taking the ambient gradient $\overline \nabla_X$ of $\Omega$ with respect to any tangent vector $X$ yields the equality:
		\[\overline{\nabla}_X \Omega = -i d\theta(X) \cdot \Omega + i H(X).\]
		We refer the reader to Thomas--Yau \cite{Thomas2002} for the full calculation, which is first attributed to Oh \cite{Oh1994}. Since $\Omega$ is parallel, the result follows.
	\end{proof}
	
	Note that the above proof implies that for any holomorphic volume form $\Omega$ defining a Lagrangian angle by $\Omega_L = e^{i\theta} \operatorname{vol}_L$, we have
	\[H(X) \cdot \Omega  = d\theta(X) \cdot \Omega -i \nabla_X \Omega\]
	for any $X \in T_p L$, even without the parallel condition.
	
	Let $\{L_\alpha\}_{\alpha \in I}$ be a Lagrangian torus fibration of a subset $U \in M$ of a K\"ahler--Einstein manifold. For each $\alpha$, define a holomorphic volume form $\Omega_{L_\alpha}$ along $L_\alpha$, i.e. a unit section of the canonical bundle $K_{M|_{L_\alpha}}$, by 
	\begin{align*}
	\Omega_L(X_1,\dots,X_n) &= \operatorname{vol}_L(X_1,\dots,X_n), \\
	\Omega_L(JX_1,X_2,\dots,X_n) &= i \operatorname{vol}_L(X_1,\dots,X_n), \quad \text{etc.}
	\end{align*}
	for tangent vectors $X_i \in T_p L$. Now let $x\in U$. There is a unique $\alpha(x) \in I$ such that $x \in L_{\alpha(x)}$, so we define a section of $K_{M|_U}$ by 
	\[\Omega(x)(X_1,\dots,X_n) = \Omega_{L_{\alpha(x)}} (X_1,\dots,X_n),\]
	for $X_i \in T_x M$. We call $\Omega$ a relative holomorphic volume form (to the fibration $L_\alpha$). 
	
	In contrast to the Calabi--Yau case where $\Omega$ was always parallel, the volume form defined here is in general not parallel. In \cite{Lotay2018}, the form $\Omega_L$ is differentiated in tangent and normal directions. For tangent vector fields $X \in \Gamma(T L)$ we have
	\begin{equation}\label{Omega tangent deriv}	\overline \nabla_X \Omega_L = i H_L(X) \, \Omega_L \end{equation}
	where $H_L$ is the mean curvature 1-form on $L$. On the other hand, if $JY = \frac{\partial A_s}{\partial s}|_{s=0}$ is the normal vector field corresponding to a 1-parameter family $A_s:L \to M$ of Lagrangians immersions then the normal derivative is
	\begin{equation}\label{Omega normal deriv}	\overline \nabla_{JY} \Omega_L = -i \operatorname{div}_{L}(Y) \, \Omega_L.\end{equation}
	Now suppose the fibration $\{L_\alpha\}$ are the level sets of a moment map for an isometric Hamiltonian $T^n$-action on $U$. Since the action is an isometry, any vector field $\tilde X$ generated by the subgroup $\exp(tX)$ of $T^n$ has $\mathcal L_{\tilde X} \operatorname{vol}_{L_\alpha} = \operatorname{div}_{L_\alpha}(\tilde X) \operatorname{vol}_{L_\alpha} = 0$. Furthermore, since the action is Hamiltonian, $J\tilde X$ is a normal vector field corresponding to a 1-parameter family of Lagrangian immersions. So we have shown the following:
	
	\begin{thm}\label{thm-intro-H=dtheta}
		Let $\Omega$ be a holomorphic volume form on an open subset $U \subset M$ of a K\"ahler--Einstein manifold. If $L$ is a Lagrangian in $U$, then for any $X \in T_p L$ we have
		\[H(X) \cdot \Omega  = d\theta(X) \cdot \Omega -i \nabla_X \Omega,\]
		where $H$ is the mean curvature 1-form and $\theta$ is the Lagrangian angle of $L$ with respect to $\Omega$. 
		
		Suppose now that $M$ is an \textit{isometric toric manifold}, that is to say there is an isometric Hamiltonian action of $T^n$ on $M^{2n}$. Away from the singular points of the action, the level sets $\{L_\alpha\}$ are a Lagrangian fibration, and we can define $\Omega$ such that $\theta(L_\alpha) =0$ for all $\alpha$. Then for any $X=Y+JZ$ with $Y,Z \in T_p L$ we have
		\[H(X) = d\theta(X) + H_{L_\alpha(p)} (Y),\]
		where $H_{L_\alpha(p)}$ is the mean curvature 1-form of the unique Lagrangian $L_\alpha$ passing through $p$. 
	\end{thm}
	
	This allows us to relate the curvature of a Lagrangian to the curvature of a fibration via the Lagrangian angle. Choosing a fibration with easily calculable curvature, this will vastly simplify the calculation of mean curvature. We use this technique in Section \ref{sec-equiv}.
	
	\subsection{Evolution equations}
	
	The following calculation appears in Thomas--Yau \cite{Thomas2002}, but the results were known by Oh \cite{Oh1994} and Smoczyk \cite{Smoczyk1999}. For any holomorphic volume form $\Omega$ defining a Lagrangian angle by $\Omega_L = e^{i\theta} \operatorname{vol}_L$, we have
	\begin{align*}
	i \frac{\partial}{\partial t} \theta e^{i\theta} \operatorname{vol}_L + e^{i\theta} \frac{\partial}{\partial t}\operatorname{vol}_L =& \frac{\partial}{\partial t} \left( e^{i\theta} \operatorname{vol}_L \right) = \mathcal L_{V}  \Omega_L = d\left( V  \lrcorner  \Omega_L \right) = -i \, d\left(e^{i\theta} (JV) \lrcorner \operatorname{vol}_L \right) \\
	=& e^{i\theta} d\theta \wedge \left(JV \lrcorner \operatorname{vol}_L \right)  -i e^{i\theta} d^\dagger (JV) \operatorname{vol}_L
	\end{align*} 
	In Calabi--Yau manifolds, we have that $\vec H = J\nabla \theta$, and hence under mean curvature flow where $V = \vec H$ we have the evolution equations
	\begin{equation}\label{eqn-ev-theta-vol}
	\begin{split}
	\frac{\partial}{\partial t} \theta &= d^\dagger d \theta = \Delta \theta \\
	\frac{\partial}{\partial t} \operatorname{vol}_L &= -|\vec H|^2 \operatorname{vol}_L \, .
	\end{split}
	\end{equation}
	The mean curvature 1-form $H$ satisfies the evolution equation
	\begin{equation}\label{H-ev}\frac{\partial}{\partial t} H = d d^\dagger H + \kappa H,\end{equation}
	where $\kappa$ is the Einstein constant, i.e. $ \rho = \kappa  \omega$. It is clear then that the cohomology class $[He^{-\kappa t}]$ is preserved under the flow. In particular, $H$ exact is preserved.
	
	\subsection{The Cieliebak--Goldstein theorem}\label{sec-CG-thm}
	
	Recall that the space of Lagrangian subspaces $\mathcal L(n)$ in $\bR^{2n}$ is isomorphic to $U(n)/O(n)$, and hence $\det^2$ induces an isomorphism from $\mu:\pi_1(\mathcal L(n)) \to \bZ$, called the Maslov index. The Maslov class of a disc is defined to be the Maslov index of the boundary under any local trivialisation. Then we have the following theorem of Cieliebak--Goldstein \cite{Cieliebak2004}, which is fundamental to the rest of this paper:
	\begin{thm}[Cieliebak--Goldstein]\label{thm-CG-OG}
		In a K\"ahler--Einstein manifold $M$ with Einstein constant $\kappa$, the mean curvature 1-form $H$ of a Lagrangian $F:L\to M$ is related to the Maslov class $\mu$ of a disc $u:(D,\partial D) \to (M,L)$ by \footnote{Here and throughout the rest of this paper, we abuse notation by conflating forms with their pullbacks and curves in the image of a Lagrangian with their pre-image. For instance, in (\ref{CG formula}), 
			\[\int_D \omega = \int_D u^*\omega, \quad \int_{\partial D} H  = \int_{F^{-1} (u(\partial D))} H.\]}
		\begin{equation}\label{CG formula}\kappa \int_D  \omega - \pi \mu(D) = -\int_{\partial D} H. \end{equation}
	\end{thm}
	
	We call a Lagrangian submanifold monotone if for any disc $u:(D, \partial D) \to (M,L)$, 
	\begin{equation}\label{eq-mono}
	\int_D  \omega = c \mu(D),
	\end{equation}
	for a constant $c$ dependent on $M$ and $L$ but not $u$. We call a disc $u$ Maslov $m$ if the $\mu(u) = m$.  In the case of a monotone Lagrangian in an exact Calabi--Yau manifold (i.e. $ \omega = d \lambda$), and in view of (\ref{CG formula}), we see that (\ref{eq-mono}) is equivalent to 
	\[\int_{\partial D} \lambda = \int_D  \omega = c \mu(D) =  \frac{c}{\pi} \int_{\partial D} H =  \frac{c}{\pi} \int_{\partial D} d\theta,\]
	hence in the literature for Lagrangian mean curvature flow where the Calabi--Yau case (specifically $\bC^n$) is frequently the primary focus, the definition of monotone is often taken as 
	\[[\lambda] = C[d\theta].\]
	
	\begin{rem}
		The Cieliebak--Goldstein formula is a generalisation of the Gauss-Bonnet formula. When $M$ is a surface, $\omega$ is the Riemannian volume form on $M$, and so $M$ Einstein implies that \[\int_D K \operatorname{vol}(D) = \kappa \int_D \omega,\]
		where $K$ is the Gauss curvature. Moreover, all curves are Lagrangian so 
		\[\int_{\partial D} k_g \operatorname{vol(\partial D)}  = \int_{\partial D} H\]
		where $k$ is the geodesic curvature. The Euler characteristic of a disc is 1, and the Maslov class of a holomorphic disc in a symplectic surface is 1 by definition.   
	\end{rem}
	
	The above remark helps to motivate a mild generalisation of the Cieliebak--Goldstein formula to include $J$-holomorphic polygons with boundary on multiple intersecting Lagrangians, comparable to generalising Gauss--Bonnet with a smooth boundary to a piecewise-smooth boundary with corners and turning angles.
	\begin{thm}\label{thm-CG-gen}
		Let $L_1,\dots,L_m$ be Lagrangian in $M$ and let 
		\[u:(D,(\partial D_1,\dots,\partial D_m)) \to (M,(L_1,\dots,L_m))\]
		denote a map from the unit disc with $m$ marked points $p_i$ on the boundary to $M$, mapping $p_i$ to $L_i \cap L_{i-1}$ and mapping the arc $\partial D_i$ from $p_i$ to $p_{i+1}$ to $L_i$. Then 
		\begin{equation}\label{eqn-CG-gen}\kappa \int_D \omega - \pi \tilde \mu(D) = -\sum_i \int_{\partial D_i} H_{L_i}
		,\end{equation}
		where $\tilde \mu$ is the Maslov class of $u:(D,(\partial D_1,\dots,\partial D_m)) \to (M,(L_1,\dots,L_m))$, defined in the proof below.
	\end{thm}
	
	\begin{proof}
		The proof is equivalent to the generalisation of the Gauss--Bonnet formula from surfaces without corners to surfaces with corners. We refer the reader to \cite{Evans_2022} for a complete description.
	\end{proof}
	
	Let us now consider the implications of the Cieliebak--Goldstein formula for Lagrangian mean curvature flow. From (\ref{CG formula}) and the evolution equation (\ref{H-ev}) we obtain 
	\begin{equation}\label{KE-area-ev}
	\begin{split} 
	\frac{\partial}{\partial t} \int_D \omega =-\frac{1}{\kappa} \int_{\partial D} dd^\dagger H + \kappa H &= -\int_{\partial D} H \\
	&= \kappa \int_{\partial D} \omega - \pi \mu(D).
	\end{split}	
	\end{equation}
	
	$L$ is monotone when $\mu(D)$ is proportional to $\int_D \omega$, so we note two immediate corollaries for $\kappa \neq 0$.
	\begin{cor}\label{Exact is monotone}
		Let $L$ be a Lagrangian in a K\"ahler--Einstein manifold with $\kappa \neq 0$. 
		$H$ is exact if and only if $L$ is monotone with monotone constant $\pi/\kappa$.
	\end{cor}
	
	\begin{cor}\label{Monotone preserved}
		Monotone Lagrangians are preserved under mean curvature flow. When $\kappa \neq 0$, the monotone constant $\pi/\kappa$ is invariant under the flow.
	\end{cor}
	
	\begin{proof}
		The result has been shown already for $\kappa = 0$. For $\kappa \neq 0$, the result follows from Corollary \ref{Exact is monotone} and the fact that exactness of $H$ is preserved by equation (\ref{H-ev}).
	\end{proof}

	To illustrate the theory so far, we consider the best understood example of Lagrangian mean curvature flow in non-Ricci-flat manifolds. 
	
	Consider the two sphere $S^2=\bCP^1$ with the standard K\"ahler metric and let $\gamma$ be an embedded closed curve in $S^2$. Then there are, up to reparametrisation, exactly two $J$-holomorphic discs $u_1,u_2 :D \to  S^2$ with $u_i(\partial D)  = \gamma$. We have that $\gamma$ is monotone when 
	\[\int_D u_1^* \omega = \int_D u_2^* \omega,\]
	where $\omega$ is the standard Fubini--Study form on $\bCP^1  = S^2$, i.e. when $\gamma$ divides $S^2$ into two pieces of equal area. Then we have two behaviours:
	
	\begin{prop}\label{prop-mcf-spher}
		{$\ $}\\[-3ex]
		\begin{enumerate}
			\item If $\gamma$ is not monotone, $\gamma$ attains a type I singularity in finite time with blow-up a self-shrinking circle.
			\item If $\gamma$ is monotone, mean curvature flow exists for all time and converges in infinite time to a great circle.
		\end{enumerate}
	\end{prop}
	
	\begin{proof}
		Recall Grayson's theorem \cite{grayson_1989}: curve-shortening flow in surfaces either attains finite-time singularities with type I blow-up a shrinking circle, or exists for all times and converges to a geodesic. This the result follows from (\ref{KE-area-ev}) in both cases.
	\end{proof}

	\section{Type I singularities in Fano manifolds}\label{sec-monotone-type-I}
	
	We saw that monotone curves did not attain type I singularities: heuristically, any type I singularity would require the collapsing of one of the disc classes, which is prohibited by the monotone condition. We now generalise this to higher dimensions. First, we can classify all zero-Maslov self-shrinkers that may arise as a type I blow-up by a result of Groh--Schwarz--Smoczyk--Zehmisch \cite{Groh2007}: 
	\begin{thm}\label{SSzeroMaslov}
		If $F:L^n\to  \bC^n$ is a zero-Maslov Lagrangian self-shrinker arising as a result of a type I blow-up, then $L$ is a minimal Lagrangian cone. 
	\end{thm}
	This follows directly from \cite[Theorem 1.9]{Groh2007}, noting that type I blow-ups have bounded area ratios.
	
	Since type I blow-ups are smooth, embedded self-shrinkers for type I singularities, this implies there are no zero-Maslov type I blow-ups for type I singularities. Since any type I model is locally symplectomorphic to the standard unit ball, this excludes the possibility of type I singularities for monotone Lagrangians:
	
	\begin{thm}\label{Monotone-typeI}
		Let $F_t:L^n \to M^{2n}$ be a monotone Lagrangian mean curvature flow, $\kappa \neq 0$. Then $F_t$ does not attain any type I singularities.
	\end{thm}
	
	\begin{proof}
		Suppose for a contradiction that $F_t$ attains a type I singularity at time $T$. Any sequence $\eta_i \to \infty$ subsequentially defines a type I blow-up
		\[\tilde F_s := \lim_{i\to \infty} \tilde F^{\eta_i}_s = \lim_{i\to \infty} \eta_i F_{T+\eta_i^{-2}s}.\]
		Since the singularity is type I, $\tilde F(L):= \tilde F_{-1}(L)$ is a non-planar embedded Lagrangian self-shrinker, and hence by Theorem \ref{SSzeroMaslov} has non-zero Maslov class.
		
		Let $\tilde D \in H_2(\bC^n,\tilde F(L))$ have $\mu(\tilde D) > 0$. The convergence of $\tilde F_s$ to a type I blow-up is smooth and the Maslov class is topological, so for all sufficiently large $i$, there exists $\tilde D_i \in H_2(\bC^n, \tilde F_{-1}^{\eta_i})$ with $\mu(\tilde D_i) = \mu(\tilde D) > 0$ and $\tilde D_i \to \tilde D$ as $i \to \infty$. Furthermore, $\tilde D^i$ are the images under the parabolic rescaling of discs $D_i = \eta_i^{-1} \tilde D_i \in \pi_2(W,F_{T-\eta_i^{-2}}(L))$. Since $L$ is monotone and the Maslov class is invariant under rescaling,
		\[ \int_{D_i} \omega =\frac{\pi}{\kappa} \mu(\tilde D_i) = \frac{\pi}{\kappa} \mu(\tilde D) > 0\]
		for all $i$, but 
		\[\lim_{i\to \infty} \int_{D_i} \omega = \lim_{i \to \infty} \int_{\eta_i^{-1} \tilde D_i} \omega = 0,\]
		a contradiction.
	\end{proof}
		
	This theorem is the positive curvature equivalent of the result of Wang \cite{Wang2001} showing that almost-calibrated Lagrangians do not attain type I singularities in Calabi--Yau manifolds. This strengthens the perspective that monotone submanifolds are the correct class of submanifolds to study to find positive curvature analogues of the Thomas--Yau conjecture. The rest of the paper will be devoted to exploring what a Thomas--Yau conjecture looks like in the prototypical Fano surface $\bCP^2$.

	\section{Equivariant Lagrangians in \texorpdfstring{$\bCP^2$}{CP2}}	
	\subsection{Clifford and Chekanov tori in Lefschetz fibrations}
	
	Our goal is to study the behaviour of Clifford and Chekanov tori in $\bCP^2$ under mean curvature flow, but this presents a number of difficulties. The main problem is the class of potential singularities is too great. Heuristically, singular behaviour is local and since $\bCP^2$ looks flat on sufficiently small scales, we expect that a priori any singular behaviour observed for zero-Maslov Lagrangians in $\bC^2$ should also occur for monotone Lagrangians in $\bCP^2$. In particular, zero-object singularities\footnote{As a note on the terminology: The obvious surgery at such a singularity bubbles off an immersed Lagrangian sphere with a single transverse self-intersection. Since such a sphere represents a zero object in the Fukaya category, it seems sensible to call these singularities which are collapsing zero-homotopic curves zero-object singularities. See Joyce \cite[Section 3.7]{joyce_2015} for a more detailed description.} like those studied in Neves \cite[Figure 3]{Neves2007} (Figure \ref{fig-zero-obj})  can occur and currently we have little understanding about the nature of these singularities. A second issue is that there is no control over where the singularity happens and what Lagrangian cone the type I blow-up produces, even under the assumption that we obtain Lawlor neck singularities. Since these are general problems in Lagrangian mean curvature flow, we choose a symmetric subclass of Lagrangians in $\bCP^2$ which cannot have the zero-object singularities and where we have strong control over the location and type of the singularities.
	
	\begin{figure}
		\centering
		\includegraphics[scale=1]{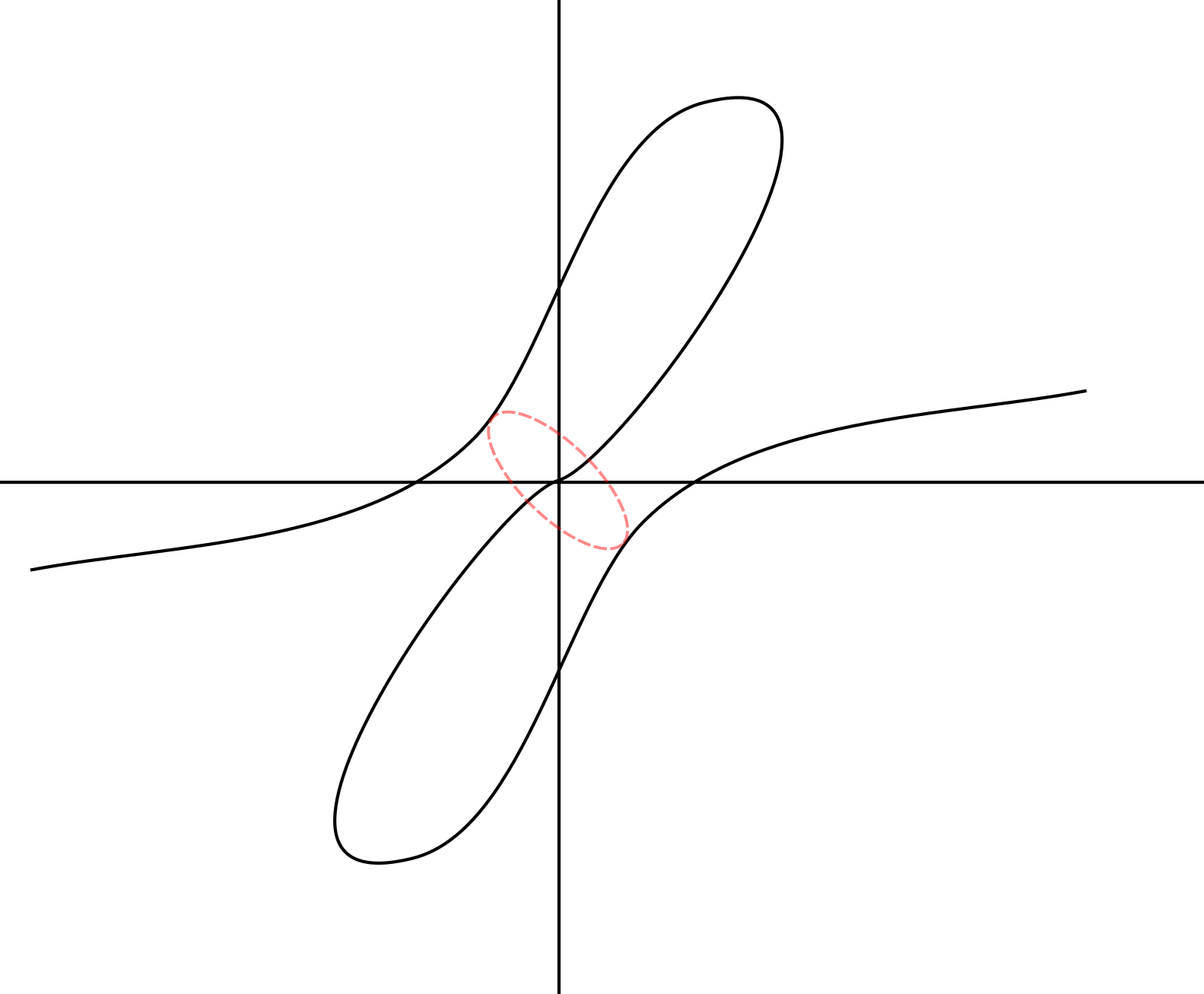}
		\caption{An example of a zero-object singularity - A Lagrangian plane attaining a type II singularity.}\label{fig-zero-obj}
	\end{figure}
	
	We consider two rational maps $\bCP^2 \to \bCP^1$. The first is the Lefschetz fibration 
	\[	f([x:y:z] ) =  [xy:z^2] \]
	in the complement of the anti-canonical divisor $D = \{(xy-z^2)z =0\}$. The second is the projection
	\[ \pi( [x:y:z] ) =  [y:z]. \]
	This extends to a foliation of $\bCP^2$ by holomorphic spheres each intersecting at a single point $[0:0:1]$ with intersection number 1.  
	
	We call a subset $U \subset \bC$ point-symmetric if  $x \in U$ if and only if $-x \in U$. For a point-symmetric curve $\gamma(s) \in \bC$, define
	\[L_\gamma^0 = \left\{ \left[ \gamma(s) e^{i\alpha}: \gamma(s) e^{-i\alpha} :1 \right] \, \big| \; \alpha \in \bR, s \in \bR \right\}\]
	and notice that since $\gamma$ is point-symmetric, $f(L_\gamma^0) = \{[\gamma(s)^2:1]:s \in \bR\}$ is an embedded curve in $\bCP^1$ if $\gamma(s)$ is embedded in $\bC$. We will also allow unions of two smooth non-intersecting curves such that the union is point-symmetric. By an abuse of notation, we refer to such a curve as $\gamma(s)$ where the parameter $s$ is now allowed to vary over two intervals or circles. 
	
	First, we identify various Lagrangians in this format. Let $\gamma(s) = s \in \bR \subset \bC$. Then 
	\[\{ [s e^{i\alpha}: se^{-i\alpha}:1] \} \]
	lies above $\gamma$ and is compactified by the circle $[e^{i\alpha}:e^{-i\alpha}:0]$ at infinity. The resulting manifold is
	\[L_{\bR}^0 :=  \{ [s e^{i\alpha}: se^{-i\alpha}:1] \} \cup \{[e^{i\alpha}:e^{-i\alpha}:0]\} =  \{[1:e^{-2i\alpha}: re^{-i\alpha}]\} \cup \{[0:0:1]\}\]
	or equivalently, using the substitution $s = \cot \phi$,
	\[L_{\bR}^0 = \{ [\cos \phi e^{i\alpha} : \cos \phi e^{-i \alpha} : \sin \phi] : \phi \in [0,\pi/2], \alpha \in \bR \}.  \]
	Note that $L_{\bR}^0$ is fixed under the anti-symplectic involution $X:[x:y:z] \mapsto [\bar y : \bar x : \bar z]$, hence is isomorphic to $\bRP^2$. The same applies for any other line through the origin in $\bC$. 
	
	The curve $\gamma_r(s) = re^{is}$ lifts to a Lagrangian torus of Clifford-type, which is monotone and minimal if and only if $r=1$. Furthermore, any point-symmetric closed curve enclosing the origin lifts to a torus of Clifford-type, monotone if and only if the symplectic area contained is equal to $4\pi/6 = 2\pi/3.$ This follows from the Cieliebak--Goldstein formula (\ref{CG formula}), $\kappa = 6$ and the fact that the disc is Maslov 4. Any closed circle $\gamma$ not enclosing the origin and its point-symmetric image $-\gamma$ together lift to a torus of Chekanov-type (provided $\gamma$ does not intersect $-\gamma$), monotone if and only if the area contained is $2\pi/6 = \pi/3$. The fact that these Lagrangians are Clifford and Chekanov respectively can be checked by observing their images under the Lefschetz fibration $f$ and comparing with the standard definitions in Auroux, for instance, \cite{Auroux_2007}.
	
	We will distinguish between Clifford tori and Chekanov tori by their intersections with real projective planes $\bRP^2$. Immediately we observe that any closed curve $\gamma$ enclosing the origin intersects any line $l$ through the origin in at least two points, hence any equivariant Clifford torus $L_\gamma \cong L_{\text{Cl}}$ intersects $L_l \cong \bRP^2$ in at least one circle. Indeed, this result is generalisable: $L_\text{Cl}$ is non-displaceable from $\bRP^2$, as can be shown in multiple different ways (see for instance \cite{biran_cornea_2009} or \cite{entov_polterovich_2009}). Indeed, Amorim and Alston \cite{alston_amorim_2011} give a lower bound of 2 for the number of intersections between a Clifford torus and $\bRP^2$. On the other hand, one can easily observe that there exists a pair of point-symmetric circles $\gamma(s) \in \bC$ each containing a disc of area 2 and not intersecting the imaginary axis $i\bR \in \bC$. Hence Chekanov tori are displaceable from $\bRP^2$.
	
	In the sequel, it will be useful to consider cones of real projective planes intersecting our flowing Lagrangian tori, so we make the following definition:
	\begin{defn}
		Denote by $l_b$ the line $\{se^{ib}:s \in \bR\} \subset \bC$. For $a \in (0,\pi)$, let $C^b_a$ be a cone of opening angle $a$ about $l_b$, i.e.  the union of $l_{b-a/2}$ and $l_{b +a/2}$. We say that a point-symmetric pair of closed curves $\gamma$ is contained in $C^b_a$ if $\arg(\gamma(s)) \in (b-a/2,b +a/2) \cup (-b-a/2 , -b +a/2)$ for all $s$. 
	\end{defn}
	Finally, we define the symmetry condition we will be using. 
	\begin{defn}
		A Lagrangian $L_\gamma$ is called equivariant if $\gamma$ is point-symmetric and $\bZ_2$-symmetric with respect to the real axis.
	\end{defn}
	The point-symmetry is an $S^1$-symmetry on the level of $L_\gamma$, so the equivariance considered is an $(S^1\times \bZ_2)$-symmetry. The main reason for this symmetry condition is to greatly restrict the variety of singularities that can occur. Specifically, we want to have only Lawlor neck singularities occurring at the origin, with type I blow-up given by $C^0_{\pi/2}$. We shall see how the equivariance gives this in Section \ref{sec-sing}.
	
	\subsection{Mean curvature of equivariant Lagrangians}\label{sec-equiv} 
	
	Before proceeding to the proofs of the main theorems, we calculate the evolution equation satisfied by the profile curve $\gamma$ under mean curvature flow. Despite being the governing equation for the rest of the results in the paper, we do not need the precise formulation frequently: it is only necessary for the explicit construction of various barriers. However, the derivation of the evolution equation for $\gamma$ is interesting in its own right since we calculate the mean curvature of $L_\gamma$ by a novel method.
	
	Recall the fibration $\{L_\alpha\}$ by Clifford-type tori given by the fibres of the moment map 
	\[\mu([x:y:z]) = \frac{1}{|x|^2 +|y|^2 +|z|^2}\left( |x|^2, |y|^2\right).\]
	The equivariant fibres are
	\[L_r = \{L_{re^{i\phi}}:r >0\}\]
	and for the rest of this paper, we denote by $\Omega$ the holomorphic volume form relative to $\{L_\alpha\}$. We first calculate the mean curvature of $L_r$, then we calculate the mean curvature of any other equivariant torus $L_\gamma$ by calculating the relative Lagrangian angle between $L_\gamma$ and $L_r$ using Theorem \ref{thm-intro-H=dtheta}. Recall that Theorem \ref{thm-intro-H=dtheta} implies that the relative Lagrangian angle $\theta = \theta_{\operatorname{rel}}$ defined by $\Omega$ satisfies
	\[H_{L_\gamma}(X) = d\theta_{\operatorname{rel}}(X) + H_{L_r}(\pi X)\]
	where $\pi$ is the projection onto the tangent bundle of $L_r$. 
	
	We calculate the mean curvature 1-form of Clifford tori $L_r$ indirectly. The curve $\gamma(s) =re^{is}$ bounds a $J$-holomorphic disc which lifts to $\bCP^2$ giving a disc 
	\[u:z\mapsto [r z : r z :1]\]
	with boundary on $L_r$ of Maslov index 4. Cieliebak--Goldstein gives 
	\[-\int_{\partial D} H_{L_r} = 6 \int_D \omega - 4\pi\]
	since $\kappa = 6$ for $\bCP^2$ with the Fubini--Study metric. We calculate $\int_D \omega$ directly. We have that in radial coordinates $x= r_1 e^{i\theta_1}$, $y= r_2 e^{i\theta_2}$, the K\"ahler form is
	\begin{align*}\omega = \frac{1}{\left(1+r_1^2+r_2^2\right)^2}
	\Big(&r_1(1+r_2^2) dr_1 \wedge d\theta_1 - r_1 r_2^2 dr_1 \wedge d\theta_2\\ &- r_1^2 r_2 dr_2 \wedge d\theta_1 + r_2(1+r_1^2)dr_2 \wedge d\theta_2\Big),\end{align*}
	so 
	\begin{equation}\label{eqn-area-disc}\int_D \omega = 2\pi \int_0^r \frac{2\tilde r}{(1+2\tilde  r^2)^2} \, d\tilde r = \pi\frac{ 2r^2}{1+2r^2}.\end{equation}
	Hence using Cieliebak--Goldstein, we have
	\[-\int_{\partial D} H_{L_r} = 6\pi\frac{ 2r^2}{1+2r^2}  - 4\pi = 4\pi \left(\frac{r^2-1}{1+2r^2} \right). \]
	Note that $r=1$ is the monotone flat Clifford torus. Then by the symmetry of the tori $L_r$, we have that 
	\begin{equation}\label{eqn-HLr} H_{L_r} = -2\left(\frac{r^2-1}{1+2r^2} \right) ds\end{equation}
	as a 1-form on $L_r$.
	
	Next we calculate the relative Lagrangian angle. If $\gamma(s) = r(s) e^{i\phi(s)}$, then $L_\gamma$ is given by the embedding
	\[F_\gamma:(s,\alpha) \to \left[r(s) e^{i\phi(s)} e^{i\alpha} : r(s) e^{i\phi(s)} e^{-i\alpha}:1 \right].\]
	Identifying the tangent space of $\bCP^2$ in the coordinate patch where $z=1$ with $\bC^2$ in the obvious way, we find that 
	\[\omega\left(\frac{\partial F_\gamma}{\partial s},\frac{\partial F_\gamma}{\partial \alpha}\right) = \omega\left(\partial_{r_1} +\partial_{r_2},\partial_{\theta_1}-\partial_{\theta_2}\right) = 0,\]
	which verifies that $L_\gamma$ is Lagrangian, and furthermore, we have
	\[\Omega_{L_r} \left(\frac{\partial F_\gamma}{\partial s},\frac{\partial F_\gamma}{\partial \alpha}\right) = \Omega_{L_r}\left(-r^{-1}r' J\partial_\phi + \phi' \partial_\phi, \frac{\partial F_\gamma}{\partial \alpha}\right),\]
	where $\partial_\phi = \partial_{\theta_1} + \partial_{\theta_2}$ and we have used $J\partial_{\theta_i} = -r_i \partial_{r_i}$. Since $\frac{\partial F_\gamma}{\partial \alpha}$ is tangent to $L_r$, the Lagrangian angle $\theta$ relative to $\Omega_{L_r}$ is given by
	\[\theta = \arg\left( \phi' - ir'r^{-1}\right) = -\tan^{-1} \left( \frac{r'}{r\phi'}\right)\]
	and hence 
	\begin{equation}\label{eqn-dtheta} d\theta = \frac{-r'' r \phi' + r'^2 \phi' +r' r \phi''}{r'^2 +r^2 \phi'^2} ds.\end{equation}
	But the Euclidean planar curvature $k$ of $\gamma$ is 
	\begin{equation}\label{eqn-k} 
	\begin{split} k &= \frac{-r''r\phi' + 2r'^2 \phi' + r'r \phi'' +r^2 \phi'^3}{\left(r'^2 +r^2 \phi'^2\right)^{3/2} }\\
	&=\left( \frac{-r''r\phi' + r'^2 \phi' + r'r \phi''}{\left(r'^2 +r^2 \phi'^2\right)} + \phi' \right) \frac{1}{\sqrt{r'^2 +r^2 \phi'^2}} \end{split}
	\end{equation}
	We have that the projection of $\frac{\partial F_\gamma}{\partial s}$ onto $L_r$ is 
	\[\pi\left(\frac{\partial F_\gamma}{\partial s}\right) = \frac{\omega\left(\frac{\partial F_\gamma}{\partial s},J \frac{\partial F_r}{\partial s}\right)}{\omega\left(\frac{\partial F_r}{\partial s},J\frac{\partial F_r}{\partial s}\right)} \frac{\partial F_r}{\partial s} = \phi'\frac{\partial F_r}{\partial s} \]
	so we are led to conclude that	
	\begin{equation}\label{eqn-alpha}H_{L_r}\left(\pi\left(\frac{\partial F_\gamma}{\partial s}\right)\right) = -2\left(\frac{r^2-1}{1+2r^2} \right) \phi'.\end{equation}
	Combining the above equations, we obtain
	\begin{align*} H_{L_\gamma} = d\theta + H_{L_r}(\pi(\cdot)) &= \left( k \sqrt{r'^2 +r^2 \phi'^2} - \phi' - 2\left(\frac{r^2-1}{1+2r^2}\right)\phi'\right) \, ds  \\
	&= \left( k \sqrt{r'^2 +r^2 \phi'^2} + \left(\frac{1-4r^2}{1+2r^2}\right)\phi'\right) \, ds.
	\end{align*}
	Hence we have that 
	\[\omega\left( \frac{\partial F_\gamma}{\partial s}, \vec H_{L_\gamma}\right) = k \sqrt{r'^2 +r^2 \phi'^2} +\left(\frac{1-4r^2}{1+2r^2}\right)\phi',\]
	but 
	\begin{align*}
	\omega\left(\frac{\partial F_\gamma}{\partial s}, J \frac{\partial F_\gamma}{\partial s}\right) &= \omega\left( r'(\partial_{r_1} +\partial_{r_2}) + \phi' (\partial_{\theta_1} +\partial_{\theta_2}), r'r^{-1}(\partial_{\theta_1} +\partial_{\theta_2}) - r\phi'(\partial_{r_1} +\partial_{r_2})\right)\\
	&= \left(r'^2 r^{-1} + r\phi'^2\right) \omega\left(\partial_{r_1} +\partial_{r_2}, \partial_{\theta_1} +\partial_{\theta_2}\right) \\
	&= 2\frac{r'^2 +r^2 \phi'^2}{(1+2r^2)^2}.
	\end{align*}
	So we conclude that
	\[\vec H_{L_\gamma} = \frac{1}{2}\left(1+2r^2\right)^2\left( k + \left(\frac{1-4r^2}{1+2r^2} \right) \frac{\phi'}{\sqrt{r'^2+r^2\phi'^2}}\right) DF_\gamma(\nu)\]
	where $\nu$ is the Euclidean normal to $\gamma$ in $\bC$. Since $\langle \gamma, \nu\rangle = -r^2 \phi' /|\gamma'|$, we have that the mean curvature flow of $F_\gamma$ in $\bCP^2$ induces an equivariant flow on $\gamma$ given by 
	\begin{equation}\label{eqn-equiv-mcf}\frac{\partial \gamma}{\partial t} = \frac{1}{2}\left(1+2r^2\right)^2\left( k - \left(\frac{1-4r^2}{1+2r^2} \right) \frac{\langle\gamma,\nu\rangle}{r^2}\right) \nu \end{equation}

	\subsection{Triangle calculations using Cieliebak--Goldstein}\label{sec-triangle}
	
	In order to prove the main results of this paper, we apply the generalised Cieliebak--Goldstein theorem (Theorem \ref{thm-CG-gen}) to certain $J$-holomorphic polygons with boundary on flowing Lagrangians. The most important are triangles with one vertex at the origin. Since these triangle calculations are ubiquitous and essential in the sequel, we review the methods involved here.
	
	\begin{ex}\label{ex-triangle}
		Let $L_\gamma$ be an equivariant Lagrangian in $\bCP^2$ intersecting the cone $C^0_{\psi}$ at points $p^+, p^-$, see Figure \ref{fig-triangle}, with Euclidean turning angle $\xi$ at $p^+,p^-$. Consider the $J$-holomorphic triangle $P$ with boundary on $L_\gamma$ given by the horizontal lift of the Euclidean triangle (also denoted $P$) with boundary on $\gamma$, $C^0_{\psi}$ and vertices at $0,p^+$ and $p^-$.
		
		\begin{figure}
			\centering
			\includegraphics[scale=1.5]{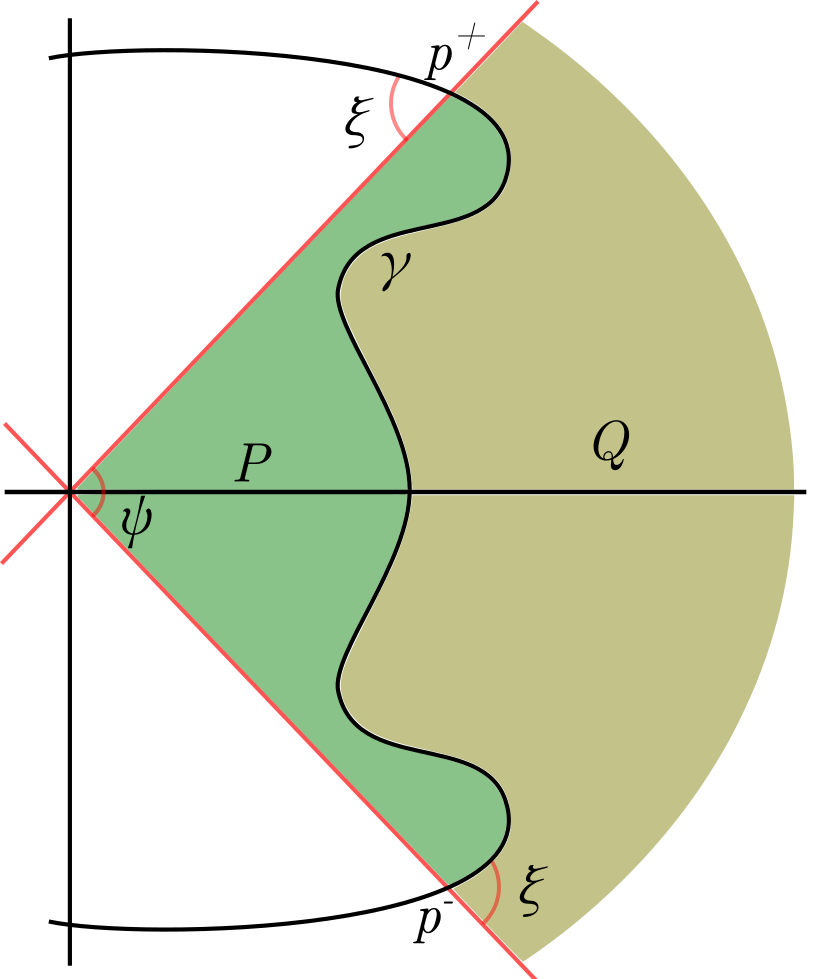}
			\caption{The triangles $P$ and $Q$ considered in Example \ref{ex-triangle}}\label{fig-triangle}
		\end{figure}
		
		We first calculate $\tilde \mu(P)$. The Maslov number can be broken down into two components: a component from the topology of the triangle, and a component from the angles at the vertices. By smoothing the corners of the triangle so that the resulting triangle does not intersect the origin, we see that the topological component is 2. The contribution from the turning angles at the either of the corners $p^\pm$ is given by $(- \theta_1(p^\pm) + \theta_2(p^\pm))/\pi$, which is equivalent to the difference in Euclidean Lagrangian angle between the Lagrangian planes $T_{p_1} L_1$ and $T_{p_1} L_2$. So we have 
		\[\tilde \mu(P) = 2 - \frac{2}{\pi} \xi - A(\psi),\]
		where $A(\psi)$ is some function of the opening angle at the origin to be determined. 
		
		We could calculate this directly by calculating the difference in Lagrangian angle between $l_{\psi/2}$ and $l_{-\psi/2}$. For the purposes of intuition however, we calculate indirectly using the example where $\gamma(s) = e^{is}$ is the minimal Clifford torus. We have that $\xi = \pi/2$, so 
		\[\tilde \mu(P) = 1 - A(\psi).\]
		Furthermore, the area of $P$ is given by 
		\[\int_P \omega = \frac{\psi}{2\pi} \frac{4\pi}{6} = \frac{\psi}{3}\]
		since the area is $4\pi/6$ when $\psi = 2\pi$. Since $H_{L_\gamma} = 0$,  (\ref{eqn-CG-gen}) implies that 
		\[A(\psi) = 1 - \frac{\kappa}{\pi} \int_P \omega = \frac{1}{\pi}\left( \pi -2\psi\right).\]
		Since the contribution of $\psi$ at the origin is independent of the choice of $\gamma$, we have that
		\begin{equation}\label{eqn-mu-triangle-origin}
		\tilde \mu(P) = 2 - \frac{2}{\pi} \xi - \frac{1}{\pi}\left( \pi -2\psi\right).
		\end{equation}
		In the important special case where $\xi = \pi$, i.e. $\gamma$ is tangent to the cone $C^0_\psi$ at the points $p^+$ and $p^-$, the sign of $\tilde \mu(P)$ is controlled by the opening angle $\psi$. We have that 
		\[\tilde \mu(P) = -\frac{1}{\pi}\left( \pi -2\psi\right)\]
		and hence $\tilde \mu(P)$ is negative for $\psi <\pi/2$ and positive for $\psi > \pi/2$.
	\end{ex}
	
	\subsubsection{Evolution equations for polygons}
	
	Since they are important in the sequel, we recall the key formulae concerning $H$ and $\theta$. By Theorem \ref{thm-intro-H=dtheta}, we have that
	\[H = d\theta + \alpha,\]
	where $\alpha$ is the 1-form $H_{L_r} (\pi (\cdot))$, where $\pi$ is projection to the tangent bundle of $L_r$. Furthermore, $\theta$ defined in this way satisfies the evolution equation 
	\[\frac{\partial}{\partial t} \theta = \Delta \theta +d^\dagger \alpha,\]
	by the same calculation that yielded (\ref{eqn-ev-theta-vol}), and the mean curvature 1-form $H$ satisfies
	\[\frac{\partial}{\partial t} H = dd^\dagger H +\kappa H.\]
	
	Recall that for a polygon $P$ with no corners, the Maslov number is the Maslov class and is invariant under mean curvature flow, and so we have 
	\[\frac{\partial}{\partial t} \int_P \omega = -\frac{1}{\kappa} \frac{\partial}{\partial t} \int_{\partial P} H = - \frac{1}{\kappa} \int_{\partial P} dd^\dagger H + \kappa H = \kappa \int_P \omega - \pi \mu(P).\]
	It initially seems reasonable to conjecture then that for a polygon $P$ with corners,
	\[\frac{\partial}{\partial t} \int_P \omega = \kappa \int_P \omega - \pi \tilde \mu(P).\]
	However, this does not hold for two reasons. Firstly, we obtain boundary terms from integrating $dd^\dagger H$. Secondly, when differentiating, we must account for potential tangential motion of the vertices of the polygon under mean curvature flow.
	
	For these reasons, we only consider the evolution equations in the context of Example \ref{ex-triangle}. We note that in this case we have that the sides of the triangle on the cone are constant angle and minimal.
	
	To that end, let $L_\gamma$ be a flowing equivariant Lagrangian, intersecting the cone $C^0_{\psi}$ at points $p^\pm$, forming a triangle $P$ as in Example \ref{ex-triangle}. Initially, we assume the intersections are transverse. Writing $\theta$ for the relative Lagrangian angle of $L_\gamma$ and $H = H_{L_{\gamma}}$ for the mean curvature 1-form, by differentiating (\ref{eqn-CG-gen}) we obtain
	\begin{align*}\frac{\partial}{\partial t} \int_P  \omega &= \frac{\partial}{\partial t}\left( \frac{\pi}{\kappa} \tilde \mu(P) - \frac{1}{\kappa} \int_{\gamma}  H\right) \\
	&= \frac{1}{\kappa}\frac{\partial}{\partial t} \left(\theta(p^-) - \theta(p^+)  \right) -\frac{1}{\kappa} \frac{\partial}{\partial t}  \int_{\gamma}  H \end{align*}
	From each term, we obtain a normal and tangential term to account for the tangential movement of the intersection points $p^\pm$ along $C^0_\psi$ under the flow. Writing the mean curvature flow as
	\[\frac{\partial}{\partial t} X = \vec H  + V\]
	for a tangential diffeomorphism $V$ to be determined, we have that 
	\[\frac{\partial}{\partial t} \left(\theta(p^\pm)\right) = \Delta \theta(p^\pm)  + d^\dagger \alpha(p^\pm) +\langle \nabla \theta, V \rangle(p^\pm),\]
	and
	\begin{align*}\frac{\partial}{\partial t} \int_{\gamma} H =& \int_{\gamma} \left( dd^\dagger H + \kappa H\right) + \left\langle \nabla \int_{\gamma} H, V\right\rangle \\
	=& \kappa \int_{\gamma} H - \Delta \theta(p^-) + \Delta \theta(p^+) - d^\dagger \alpha(p^-) + d^\dagger \alpha(p_2) \\ 
	&- \langle \nabla \theta,V\rangle (p^-) + \langle \nabla \theta,V\rangle (p^+) - \alpha(V)(p^-) + \alpha(V)(p^+)
	\end{align*}
	where we have used that $H = d\theta + \alpha$ and hence $d^\dagger H = \Delta \theta + d^\dagger \alpha$, where $\alpha$ is the closed 1-form on $L$ defined by $\alpha(X) = H_{L_r}(\pi X)$. Combining the above equations and applying (\ref{eqn-CG-gen}), we obtain	
	\begin{equation}\label{eqn-ev-P-1}
	\begin{split}
	\frac{\partial}{\partial t} \int_P  \omega =& \kappa \int_P \omega - \pi \tilde \mu(P) \\ &+ \frac{1}{\kappa} \left(-  \alpha(V)(p^-) + \alpha(V)(p^+)\right). \end{split}\end{equation} 	
	Since the intersection is transversal, we can write the tangential vector field $V$ as $\vec H+V = W$, for some vector field $W$ on $L_\gamma$ tangent to $C^0_\psi$. The vector field $W$ then gives the motion of $p^\pm$ along the cone, and we have that  
	\[\alpha(V) = \alpha(W-\vec H) = \alpha(-\vec H).\]
	Note that while $V$ is not well-defined when the intersection is not transversal, $\alpha(-\vec H)$ is well-defined everywhere on $L_\gamma$. Thus it is tempting to claim that 
	\begin{equation}\label{eqn-ev-P-2}
	\begin{split}
	\frac{\partial}{\partial t} \int_P  \omega =& \kappa \int_P \omega - \pi \tilde \mu(P) \\ &+ \frac{1}{\kappa} \left(-  \alpha(-\vec H)(p^-) + \alpha(- \vec H)(p^+)\right). \end{split}\end{equation} 
	even when the intersection is non-transversal. The most important case of this is characterised in the following lemma, where $\psi$ is a local maximum opening angle, allowed to vary in time.
	
	\begin{lem}\label{lem-tangential-triangle}
		Let $L_\gamma$ be an equivariant Lagrangian mean curvature flow in $\bCP^2$ on a time interval $[t_1,t_2]$, with $\gamma$ not passing through the origin. Suppose that for $t \in [t_1,t_2]$, $L_\gamma$ has a local maximum opening angle $\psi(t)$ on $[t_1,t_2]$, where $\psi(t)$ is a smooth function of $t$. Then the triangle $P$ defined by the cone $C^0_\psi$ and $\gamma$, with vertices at $p^\pm$ and the origin, satisfies  
		\begin{equation}\label{eqn-ev}\frac{d}{d t} \int_P \omega 
		\leq \kappa \int_P \omega +(\pi-2\psi) \end{equation}
	\end{lem} 
	
		\begin{proof}		
		Let $\gamma(s)$ be parametrised by some variable $s$. Then there exists a smooth function $S(t)$ such that $\gamma(S(t))$ attains the maximum opening angle $\psi(t)$.
		
		\begin{figure}
			\centering
			\includegraphics[scale=0.9]{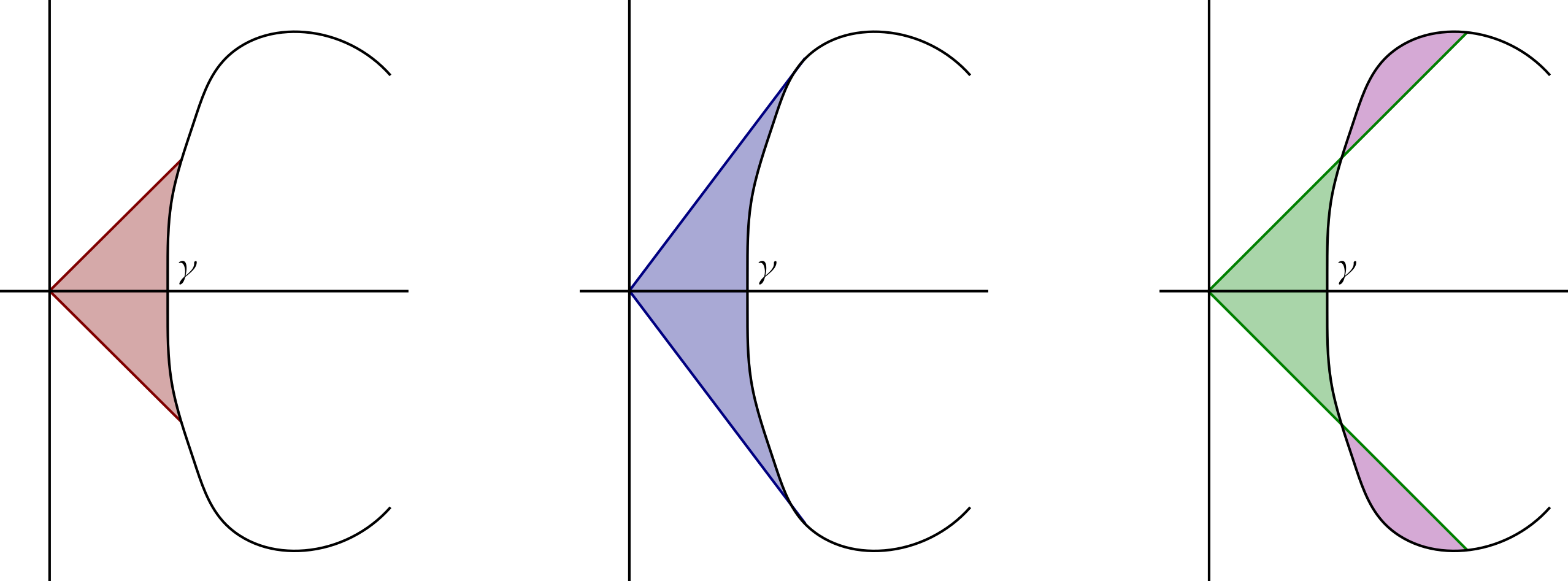}
			\caption{Three triangles $P_s$, ordered with increasing $s$. In the right diagram, the area is counted with sign, i.e. the green area is counted positively and the purple area negatively.}\label{fig-signedarea}
		\end{figure}
		
		Let $A(s,t) = \int_{P_s} \omega$, where $P_s$ is the triangle intersecting $\gamma$ at $\gamma(s)$. Here, the integral is the signed integral of $\omega$, see Figure \ref{fig-signedarea}. Then we have to calculate the time-derivative of $A$ at $S(t)$ for $t \in (t_1,t_2)$. By choosing a sufficiently small time neighbourhood $(t_-,t_+)\subset (t_1,t_2)$ of $t$, we can find a time-independent space neighbourhood $(s_-,s_+)$ of $S(t)$ for all $t$ such that $\gamma(s)$ intersects the cone transversally for all $s\neq S(t)$.
		
		For any fixed opening angle $\chi$ with transversal intersections with $\gamma$ at $p^\pm_\chi$, we have that
		\[\frac{\partial}{\partial t} \int_{P_\chi} \omega = \kappa \int_{P_\chi} \omega - \pi \tilde \mu(P_\chi) + \frac{1}{\kappa}\left( - \alpha(-\vec H)(p_\chi^+) + \alpha(- \vec H)(p_\chi^-)\right),\]
		where $P_\chi$ is the triangle of opening angle $\chi$, again calculated with sign.	Now allowing that the opening angle $\chi = \chi(s,t)$ may evolve with $s$,
		\[\frac{d}{d t} A(s,t) = \frac{\partial}{\partial t} \int_{P_\chi} \omega + \frac{d\chi }{dt} (s,t) \frac{\partial A}{\partial \chi} (s,t) \]
		and taking limits as $s \to S(t)$ gives
		\begin{align*}\frac{d}{d t} A(S(t),t) = & \kappa \int_{P_\chi} \omega - \pi \tilde \mu(P_\chi) + \frac{1}{\kappa}\left( - \alpha(-\vec H)(p_\chi^+) + \alpha(- \vec H)(p_\chi^-)\right)\\ 
		& + \frac{d\chi }{dt} (S(t),t) \frac{\partial A}{\partial \chi} (S(t),t) + \frac{d S}{dt}(t) \frac{\partial A}{\partial s} (S(t),t).
		\end{align*}
		But since $S(t)$ is a local maximum of the area by assumption, we have that
		\[\frac{\partial A}{\partial s} (S(t),t) = 0.\]
		Furthermore, the maximum opening angle is decreasing in time, so 
		\[ \frac{d\chi }{dt} (S(t),t) \leq 0,\]
		and $A$ is always increasing in $\chi$ for $\chi <\psi$, so 
		\[\frac{d\chi }{dt} (S(t),t) \geq 0.\]
		Finally, 
		\[-\alpha(-\vec H)(p^+)+\alpha(-\vec H)(p^-) < 0\]
		since the direction of the mean curvature is fixed by the assumption that $p^\pm$ are at the maximum opening angle. Since the Maslov number satisfies $\pi \tilde \mu(P) = -(\pi- 2\psi)$, we conclude that
		\[\frac{d}{d t} A(S(t),t) \leq  \kappa \int_{P_\chi} \omega + (\pi- 2\psi),\]
		as desired.
	\end{proof}
	
	\subsection{Minimal equivariant Lagrangians}\label{sec-minimal}
	
	The main result of this section is Theorem \ref{thm-main-minimal}, which we slightly expand on now that we have the relevant terminology from Section \ref{sec-equiv}.
	
	\begin{thm}
		There exists a countably infinite family of complete immersed minimal equivariant equivariant Lagrangians. In particular, given any radius $R$ with $0<R<1$, there exists a complete immersed minimal equivariant torus $L_\gamma$ with $0< \min_\gamma r(\gamma)  \leq R$, where $r$ is the Euclidean radius function on $\gamma$. 
	\end{thm}

	Since the method of proof is rather long and calculational, we provide an abridged version, referring the reader to the author's doctoral thesis \cite{Evans_2022}[Section 4.6] for full details. 
	
	From equation (\ref{eqn-equiv-mcf}), any minimal equivariant Lagrangian must satisfy
	\begin{equation}\label{eqn-min-equiv}
	k - \left(\frac{1-4r^2}{1+2r^2} \right) \frac{\langle\gamma,\nu\rangle}{r^2} = 0.
	\end{equation}
	Away from the origin, equation (\ref{eqn-min-equiv}) is a non-linear 2nd order ODE. Given any point $x \in \bC$ and an initial velocity $v \in T_x \bC$, there is a unique local solution to (\ref{eqn-min-equiv}) passing through $x$ with velocity $v$. The proof is identical to the equivalent statement for existence and uniqueness of geodesics.
	
	Two classes of solutions to (\ref{eqn-min-equiv}) are immediately apparent. First, either from the derivation of (\ref{eqn-min-equiv}) or by direct calculation, one can see that the Clifford torus  $L_1 := L_{e^{is}}$ given by the unit circle is a minimal submanifold. Second, any straight line through the origin $l_b = \{ se^{ib} : s\in\bR\}$ has $k=0$ and $\langle l_b,\nu\rangle=0$, and hence gives a minimal submanifold of $\bCP^2$, topologically a real projective plane. Furthermore, the existence and uniqueness implies that if a solution $\gamma$ at any point has $\langle\gamma,\nu\rangle = 0$, then it is a line $l_b$ everywhere.
	
	We now restrict attention to point-symmetric solutions that are graphs over sections of the unit circle, i.e. $\gamma(s) = r(s)e^{is}$ with $r(s) \in (0, \infty)$.
	
	From (\ref{eqn-min-equiv}), we have that $r$ satisfies
	\[-r''r +2r'^2 +r^2 + \left(\frac{1-4r^2}{1+2r^2}\right)(r'^2 +r^2) = 0.\]
	Rearranging, we obtain
	\begin{equation}\label{eqn-minimal-graph}
	-r''r +\left(\frac{3}{1+2r^2}\right)r'^2 + 2\left(\frac{1-r^2}{1+2r^2}\right)r^2 = 0.\end{equation}
	From here, we can derive a first integral of the equation by use of the substitution $f(s) = \log(r(s))$. Skipping the derivation, one can simply observe that 
	\begin{equation}\label{eqn-BfC} f'^2 = C\frac{e^{4f}}{\left(1+2e^{2f}\right)^3} -1 =: B(f,C).\end{equation}
	is a first integral of equation (\ref{eqn-minimal-graph}).
	
	The theory of roots of cubics provided by Descartes' rule of signs gives that for any $C>27$, $B(f,C)$ has exactly 1 positive and 1 negative real zeroes. After taking the exponential of $f$, this gives two zeroes $r_1$ and $r_2$ with $0< r_1 < 1 <r_2 < \infty$. Theses are the minimum and maximum values of $r$ for out solution $\gamma$. As $C \to 27$, $r_i \to 1$. Thus the Clifford torus is the solution with $C= 27$. As $C\to \infty$, $r_1 \to 0 $ and $r_2 \to \infty$ monotonically. 
	
	Next we approach the question of periodicity. Solutions are bounded between $r_1$ and $r_2$, hence they oscillate between the two with some period $\psi_C$, dependent on the constant $C$. The first integral implies that 
	\[\psi_C = 2 \int_{\log r_1}^{\log r_2} \sqrt{\frac{\left(1+2e^{2f}\right)^3}{Ce^{4f}-\left(1+2e^{2f}\right)^3}} df = \int_{r_1}^{r_2} \sqrt{\frac{\left(1+2r^2\right)^3}{Cr^4-\left(1+2r^2\right)^3}}\frac{1}{r}  dr.\]
	If we can find an integer pair $(m,k)$ and a corresponding value $C(m,k)>27$ such that
	\[m \psi_{C(m,k)} = 2\pi k,\]
	then we have found a complete solution to (\ref{eqn-minimal-graph}). Unfortunately the above integral cannot be evaluated explicitly using standard methods.
	
	We aim instead to analyse the limiting behaviour as $C \to \infty$, illustrated in Figure \ref{fig-periodconvergence}. We show the following:
	\begin{lem}\label{lem-period}
		The period $\psi_C$ converges to $3\pi/2$ as $C \to \infty$. 
	\end{lem}
	
	\begin{figure}
		\centering
		\includegraphics[scale=1]{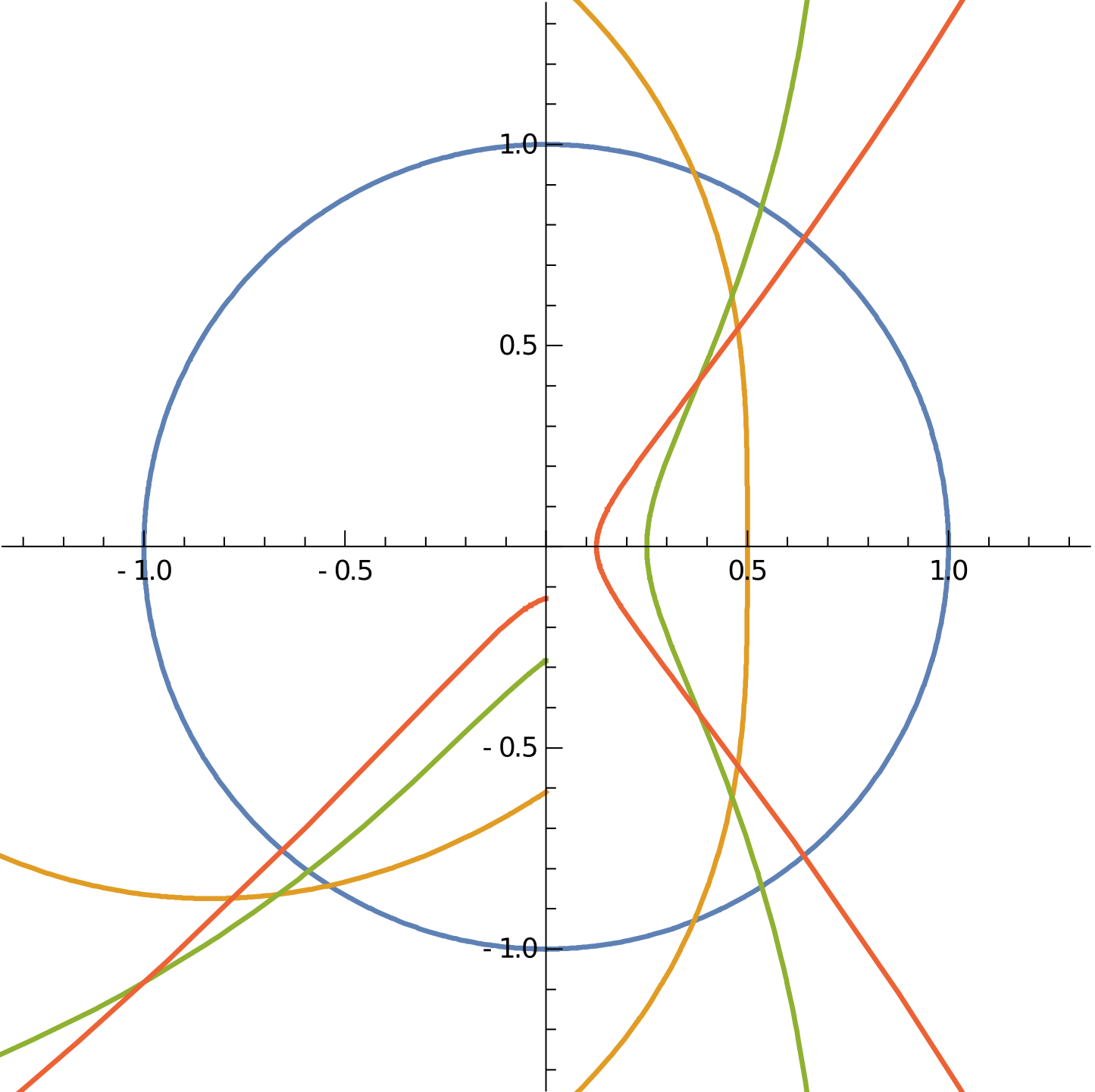}
		\caption{The inner period converges to $\pi/2$ and the outer period converges to $\pi$ as $C\to \infty$.}\label{fig-periodconvergence}
	\end{figure}
	
	\begin{proof}
		We omit many details which can be found in \cite{Evans_2022}, providing only a sketch. 
		
		\begin{enumerate}			
			\item We separate the period $\psi_C$ into two parts. The inner period $\psi_C^-$, i.e. the period where $r(s) < 1$, and the outer period, i.e. the period where $r(s) >1$. We estimate each separately using similar methods. 
		
			Beginning with the inner period, we use a geometric inequality illustrate in Figure ref{fig-uppervslower}.
			
			\begin{figure}
				\centering
 				\includegraphics[scale=0.8]{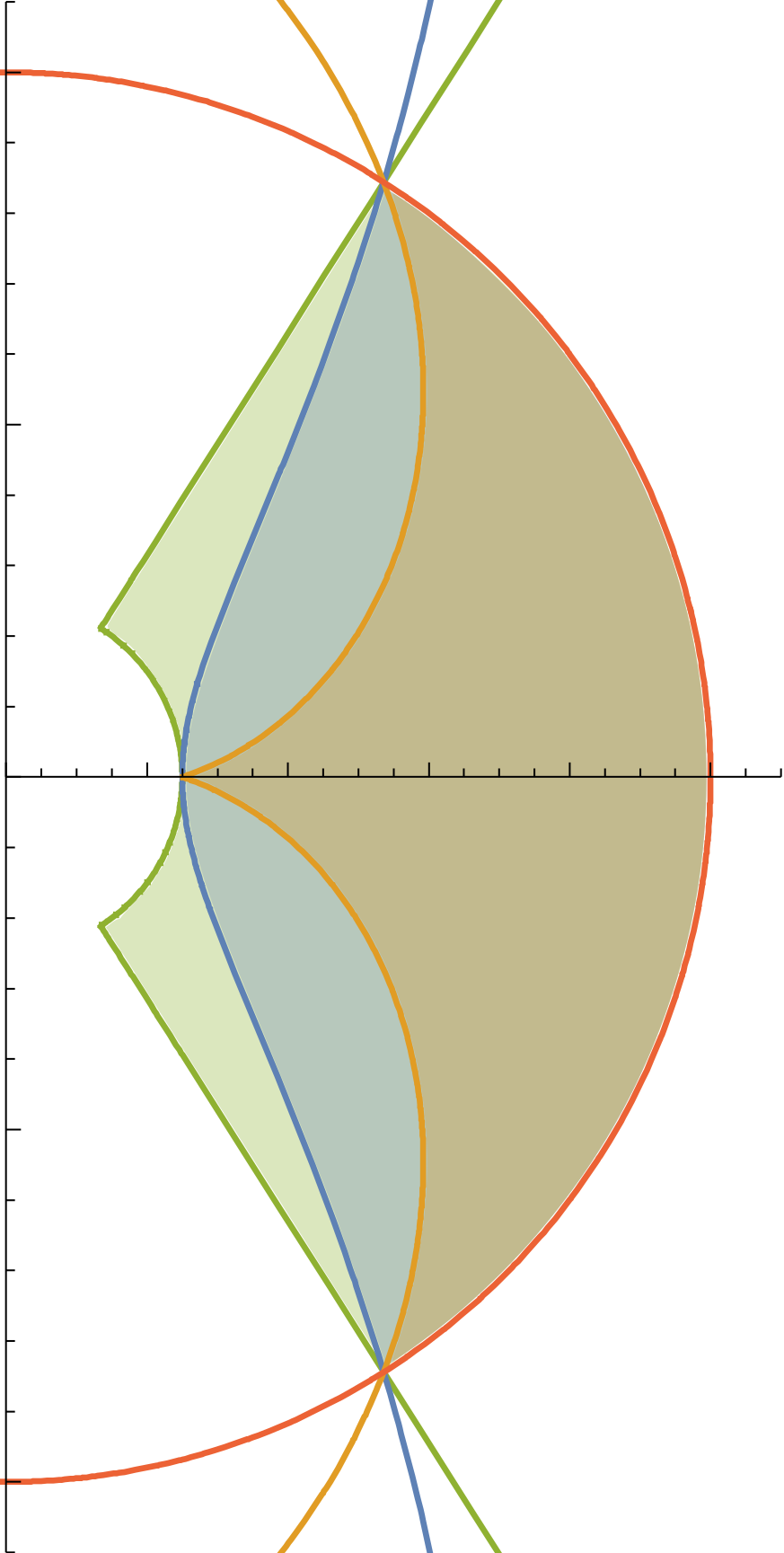}
				\caption{The area of $P$ (the blue region) is bounded above by the area of $A$ (the green region) and bounded below by the area of $B$ (the orange region).}\label{fig-uppervslower}
			\end{figure}
			
			We estimate the $J$-holomorphic biangle $P$ bounded between the unit circle and our solution $\gamma$ from above by the $J$-holomorphic quadrangle $A$ with sides on the $L_1$, $L_{r_1}$ and the cone $C^0_{\psi_C^-}$. We also estimate $P$ from below by the $J$-holomorphic triangle with sides on $L_1$ and the two ``radial straight lines'' $\eta^\pm$ joining the minimum to the intersection of $\gamma$ with $L_1$, where
			\[\eta^\pm(s) = \left(2\frac{1-r_1}{\psi_C^-}s+r_1\right)e^{\pm is}.\]
			Then we have the geometric inequalities
			\begin{equation}\label{eqn-geom-ineqs}\kappa \int_A \omega > \kappa \int_P \omega > \kappa \int_B \omega,\end{equation}
			where the second inequality holds since $r$ is convex in $s$ for $r\leq 1$.
			
			By using Cieliebak--Goldstein, we can calculate that using the upper bound that 
			\[\psi_C^-  > \frac{1+2r_1^2}{1-r_1^2}\left(\pi - \tan^{-1} \left(\sqrt{\frac{1}{27}\frac{(1+2r_1^2)^3}{r_1^4}-1}\right)\right),\]
			so $\psi_C^- > \pi/2$ for all $C>27$ and
			\begin{equation}\label{eqn-psi-lower}\lim_{C\to \infty} \psi_C^- = \lim_{r_1 \to 0} \psi_C^- \geq  \pi/2.\end{equation}
			For the lower bound, we instead calculate by performing the integral directly, and after a lengthy calculation we derive that 
			\begin{equation}\label{eqn-psi-upper} \lim_{C\to \infty} \psi_C^- = \lim_{r_1 \to 0} \psi_C^- \leq \pi/2.\end{equation}
			Combining (\ref{eqn-psi-lower}) and (\ref{eqn-psi-upper}) gives 
			\[\lim_{C\to \infty} \psi_C^- = \lim_{r_1 \to 0} \psi_C^- = \pi/2.\]
			
			\item Now we calculate the upper period, which we claim satisfies
			\[\lim_{C\to \infty} \psi_C^+ = \lim_{r_2 \to \infty} \psi_C^+ = \pi.\]
			If the geometric inequalities held, all the calculations would proceed as above and the result would follow. However, though the upper bound does hold as before, the lower bound does not a priori. This is since the $r$ is not concave for all $r>1$. However, the radius $\tilde r$ of the inflection point of $r$ is much smaller than $r_2$, i.e.
			\[ 1 < \tilde r << r_2\]
			for all $C$, and a fairly lengthy but not challenging application of Cieliebak--Goldstein yields the geometric inequality as desired. This completes the proof.			
		\end{enumerate}
	\end{proof}

	We can now prove the main theorem of this section. 
	
	\begin{thm}\label{thm-minimal-immersed-lags}
		There exists a countably infinite family of complete immersed minimal equivariant Lagrangians. In particular, for any $R>0$, there exists a complete immersed minimal equivariant Lagrangian $L_\gamma$ with $\min_{L_\gamma} r \leq R$. 
	\end{thm}
	
	\begin{proof}		
		For $C= 54$, we show that $\psi_{54} > 3\pi/2$ by an explicit calculation. We estimate the integral by separating the integrand into two parts. The first part contains no poles in the interval $[r_1,r_2]$ and hence can be estimated directly. The second part is a well-known elliptic integral which we can evaluate explicitly. The result follows. 
		
		Note that the period $\psi_C$ of a solution $\gamma_C$ to (\ref{eqn-minimal-graph}) depends continuously upon the initial condition. Since $\psi_C \to 3\pi/2$ by Lemma \ref{lem-period} and by the above $\psi_C > 3\pi/2$ for some $C>27$, we have that there exists $\delta>0$ such that for every $\psi \in (3\pi/2, 3\pi/2 +\delta)$ there exists a $C>27$ such that $\psi_C = \psi$. In particular, we can find infinitely many integer pairs $(m,k)$ and values $C(m,k)>27$ such that
		\[m \psi_{C(m,k)} = 2\pi k.\]
		Then the minimal equivariant Lagrangians $\gamma_{C(m,k)}$ described by $C(m,k)$ are complete immersed minimal equivariant Lagrangian. This is an infinitely large family of unique solutions since for every sufficiently large prime $k$, we can obtain at least one solution. 
		
		Since this argument also applies to any $\delta' < \delta$, we can construct $L_\gamma$ satisfying the second part of the theorem. 
	\end{proof}
	
	Figure \ref{fig-spirgoraphs} in the introduction illustrates the spirograph-like shape of the complete immersed minimal equivariant Lagrangians.
	
	While we are on the subject, we prove one final property of the minimal surfaces which will be useful in proving Lemma \ref{lem-scale}. The idea of the proof is similar to Lemma \ref{lem-period}, but a slightly different geometric estimate is required. Instead of estimating using a disc $B$ with boundary on a radial straight line, we use a disc $B$ with boundary on a Euclidean straight line, see Figure \ref{fig-radialvseucl}
	
	\begin{figure}
		\centering
		\includegraphics[scale=1]{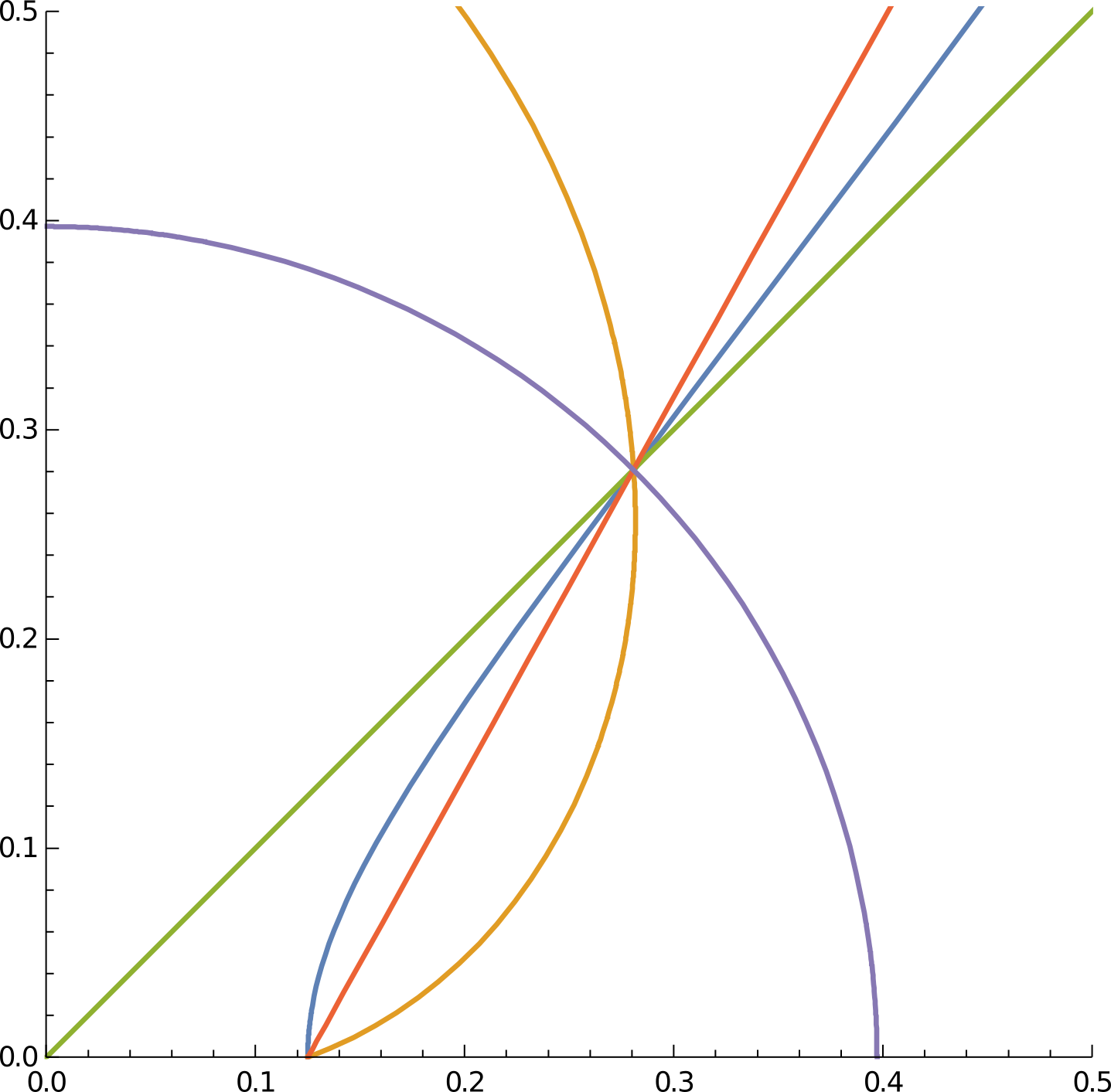}
		\caption{Radial straight lines (orange) used in Lemma \ref{lem-period} compared to the Euclidean straight line (red) used in Proposition \ref{prop-cone-radius}. Both give lower bounds for the region contained between $\gamma$ (blue) and the circle of radius $R_C$ (purple).}\label{fig-radialvseucl}
	\end{figure}
	
	\begin{prop}\label{prop-cone-radius}
		Let $\gamma_C(s) = r(s)e^{is}$, $C>27$, be a solution to equation (\ref{eqn-minimal-graph}) with initial condition $r'(0) = 0$ with $r(0) = r_1<1$. Let $R_C := r(\pi/4)$, i.e. the radius of intersection with the cone $C^0_{\pi/2}$. Then $R_C \to 0$ as $C\to \infty$.
	\end{prop}

	We need a lemma to prove Proposition \ref{prop-cone-radius}, which guarantees the Euclidean straight lines give lower bounds for sufficiently large $C>27$.	
	
	\begin{lem}\label{lem-RC<1/2}
		Let $\gamma_C$ be as in the statement of Proposition \ref{prop-cone-radius}. Then $R_C < 1/2$ for $C$ sufficiently large. 
	\end{lem}
	
	\begin{proof}
	The idea of the proof is to find a subsolution (a hyperbola) with the desired behaviour, and then use the comparison principle to obtain the result. 
	
	Consider the hyperbolas $\gamma_{a,c}$ given by 
	\[a^2 x^2 - y^2 = c^2\]
	for constants $a,c>0$. For sufficiently small $c>0$ and $a>0$ given by
	\[a^2 = \left(2c^2 +\frac{1}{2}\right) + \sqrt{\left(2c^2 +\frac{1}{2}\right)^2 + 2c^2},\]
	we have that $\gamma_{a,c}$ intersects $C^0_{\pi/2}$ at $R <1/2$. We leave the verification of this fact to the interested reader, or alternatively refer them to \cite{Evans_2022}.
	
	Let $C$ be any constant such that $\gamma_C$ has $r_1 \leq c$. We claim this implies that $\gamma_C$ intersects the cone $C^0_{\pi/2}$ at a radius less than $R < 1/2$. Suppose not. Then $\gamma_C$ intersects $\gamma_{a,c}$ at two points inside the cone $C^0_{\pi/2}$. Consider the quasilinear elliptic operator $Q(f)$ given by 
	\[Q(f) = -f'' + 2\frac{1-r^2}{1+2r^2} f'^2 + 2\frac{1-r^2}{1+2r^2}\]
	where $f= \log r$, see equation (\ref{eqn-minimal-graph}. Let $f_C$ and $f_{a,c}$ be the logarithms of the radius functions of $\gamma_C$ and $\gamma_{a,c}$ respectively. Then we have that $Q(f_C) = 0$ and $Q(f_{a,c}) < 0$. Furthermore, we have that for $s \in [-S,S] \subset (-\pi/4,\pi/4)$, $f_C(s) \leq f_{a,c}(s)$ with $f_C(S) = f_{a,c}(S)$. Since 
	\[\partial_r \left(\frac{1-r^2}{1+2r^2}\right) < 0\]
	for all $r$, we can apply the comparison principle for quasilinear elliptic operators to deduce that 
	\[f_C(s) \geq f_{a,c}(s)\]
	for all $s \in [-S,S]$, a contradiction.
	\end{proof}
	
	\begin{proof}[Proof of Proposition \ref{prop-cone-radius}]
		By Lemma \ref{lem-RC<1/2}, we have that $r_1 < R_C < 1/2$ for sufficiently large $C >27$, and hence $r(s) < 1/2$ for all $s \in [-\pi/4,\pi/4]$. Consider now $\gamma$ as a graph over the $y$-axis, i.e. $\gamma(y) = x(y) +i y$ for $y$ in an interval $I$ containing $0$. Elementary calculation gives that $x(y)$ is convex as a function of $y$ for all $r<1/2$. Hence the Euclidean straight line $\eta^\pm$ connecting the minimum value $r_1$ with $R_C e^{\pm i \pi/4}$ does not intersect $\gamma_C$ for $s \in (0,\pi/4)$. Similar to the proof of Lemma \ref{lem-period}, denote by $P$ the $J$-holomorphic biangle bounded by $\gamma$ and the circle $L_{R_C}$, and by $B$ the $J$-holomorphic triangle with boundary on $\eta^\pm$ and $L_{R_C}$.
		
		Then we have the geometric inequality 
		\[\kappa \int_P \omega > \kappa \int_B \omega.\]
		Proceeding in a similar manner to Lemma \ref{lem-period}, we find after some calculation (see \cite{Evans_2022} for details) that the geometric inequality gives
		\[\pi R_C^2 \left(\frac{3}{1+2R_C^2} - \frac{3}{2} \right) > \pi - 2 \tan^{-1}\left(\sqrt{\frac{(1+2r_1^2)^3}{r_1^4}\frac{R_C^4}{(1+2R_C^2)^3}-1}\right) -3 \sqrt 2 r_1 R_C.\]
		If $R_C$ is bounded below by $\varepsilon>0$, the right-hand side converges to 0 as $C\to \infty$ while the left-hand side is strictly less than 0, a contradiction. This completes the proof.
	\end{proof}	
	
	\section{Lagrangian mean curvature flow of equivariant Lagrangian tori in  \texorpdfstring{$\bCP^2$}{CP2}}
	
	\subsection{A Thomas--Yau-type theorem}
	
	The goal for this section is to prove Theorem \ref{thm-main-sum}. We give now an overview of the method of proof, the details of which will be contained in the following sections. First, we state a more detailed version of Theorem \ref{thm-main-sum}.
	
	\begin{thm}
		Let $L$ be an equivariant Lagrangian torus in $\bCP^2$ which intersects the cone $C^0_{\pi/2}$ a total of $4n$ times, for $n \geq 1$. With $L$ evolving by mean curvature flow, the following hold:
		\begin{enumerate}
			\item If $n = 1$, then $L$ becomes graphical over the minimal Clifford torus $L_1$ in finite time, after which $L$ exists for all time and converges to $L_1$ in infinite time. 
			
			\item If $L$ is a Chekanov torus with $n = n_0 > 1$, $L$ has a finite-time singularity at the origin. By performing a neck-to-neck surgery before the singular time, $L$ becomes a Clifford torus with $n \leq  n_0 - 1$.
			
			\item If $L$ is a Clifford torus with $n = n_0 > 1$, then after a sufficiently large time, $L$ has either has $n = 1$ or $L$ has attained a finite-time singularity at the origin. In the latter case, by performing a neck-to-neck surgery before the singular time, $L$ becomes a Chekanov torus with $n \leq n_0 - 1$.
		\end{enumerate}
	\end{thm}

	\begin{proof}
		The proof is the sum total of the results in the following sections. 
		
		\begin{enumerate}
			\item The main results of Section \ref{sec-sing} imply number of intersections with $C^0_{\pi/2}$ is strictly decreasing under the flow and  that any singularity occurs at the origin with blow-up given by $C^0_{\pi/2}$.
			
			\item In Section \ref{sec-graph-Cliff}, we prove that graphical Clifford tori exist for all time and converge to $L_0$.  
			
			\item In Section \ref{sec-Chek}, we prove that Chekanov tori always have finite-time singularities.
			
			\item In Section \ref{sec-surgery}, we define a neck-to-neck surgery procedure which by definition strictly decreases $n$, and complete the proof by dealing with the case of Clifford tori with $n > 1$.
		\end{enumerate}
	\end{proof}

	\subsection{Singularities for equivariant tori} \label{sec-sing}
	
	In this section we replicate many of the results which are known for equivariant flows in $\bC^2$ by the work of Neves (\cite{Neves2007}, \cite{Neves2010}) and Wood (\cite{Wood2019}, \cite{wood_2020})\footnote{For the readers convenience, we summarise some background information on type II singularities and blow-ups in Lagrangian mean curvature flow in Appendix \ref{appendix}.}. The corresponding section \cite{Evans_2022}[Section 4.7] comprises one of the longest and most technical sections of the paper, but the overall heuristic is simple. Since singular behaviour of a flow is a local phenomenon, and any blow-up procedure ``blows away'' the ambient curvature, we should expect any results that hold in $\bC^2$ to also hold in $\bCP^2$.

	As a first and important example of this, we show there is a direct analogue of the monotone version of Neves' Theorem B \cite{Neves2010} that holds in $\bCP^2$. 
	
	\begin{lem}\label{Theorem-B-CP2}
		Let $L$ be a monotone Lagrangian mean curvature flow in $\bCP^2$ with a finite-time singularity at $T<\infty$. For any sequence $(L_j^s)$ of rescaled flows, the following property holds for all $R>0$ and almost all $s<0$:
		
		For any sequence of connected components $\Sigma_j$ of $B_{4R}(0) \cap L_j^s$ that intersect $B_R(0)$, there exists a special Lagrangian cone $\Sigma$ in $B_{2R}(0)$ with Lagrangian angle $\bar \theta$ such that, after passing to a subsequence,
		\[\lim_{j\to \infty} \int_{\Sigma_j} f(\exp (i \theta_j^s)) \phi d \mathcal H^2 = m f(\exp(i \bar \theta)) \mu(\phi)\]
		for every $f \in C(S^1)$ and every smooth $\phi$ compactly supported on $B_{2R}(0)$, where $\mu$ and $m$ denote the Radon measure of the support of $\Sigma$ and its multiplicity respectively.
	\end{lem}
	
	\begin{proof}
		By the work of Castro--Lerma (\cite{castro_lerma_miquel_2018}) (see section \cite{Evans_2022}[Section 3.3] for a concise explanation), a monotone torus in $\bCP^2$ lifts to a monotone spherical Lagrangian $3$-torus in $S^5 \subset \bC^3$. The flow also lifts, becoming a flow with a singularity at a time $\tilde T <1/2$ (recall that $t=1/2$ is the singular time of the unit sphere $S^5$). The singularity in the lift occurs along an $S^1$, the Hopf fibre above the singular point of the original flow.
		
		At this point we can already apply \cite[Theorem A]{Neves2010} to show that we have convergence to a finite set of special Lagrangian cones with angles $\theta_k$. We want to show that we can instead apply \cite[Theorem B]{Neves2010}, which a priori only applies in $\bC^2$. Indeed, the only part of the proof of Theorem B which does not hold in higher dimensions is \cite[Lemma 5.2]{Neves2010}. So we have to show that for all $j$ sufficiently large, there exists some $C>0$ such that
		\begin{equation}\label{Nev-lem-gen}\left(\mathcal H^3(A)\right)^{2/3} \leq C \mathcal H^2(\partial A),\end{equation}
		for any open subset $A$ of $L^j_s \cap B_{6R}(0)$ with rectifiable boundary. The proof of \cite[Lemma 5.2]{Neves2010} doesn't hold immediately since the dimension is too high to apply the Michael--Simon Sobolev inequality directly in $\bC^3$. Instead, we apply the Michael--Simon Sobolev inequality for Riemannian manifolds with positive curvature to the flow in $\bCP^2$ and, by arguing that the Hopf fibre is non-collapsing at the final time, we are able to lift the resulting inequality to $\bC^3$ up to a constant. Then we can apply the Michael--Simon Sobolev inequality in $\bC^3$ and deduce the result. See \cite{Evans_2022} for details.
	\end{proof}
	
	We now state the main results
	
	\begin{prop}\label{Wood-lemmas}
		Suppose $L_\gamma$ is an equivariant monotone Clifford or Chekanov torus in $\bCP^2$. Let $T \in (0, \infty]$ be the maximal existence time for $L_\gamma$.
		\begin{enumerate}
			\item If $\gamma$ is initially embedded, it is embedded for all $t \in [0,T)$. If $L_{\gamma_1}, L_{\gamma_2}$ are two initial conditions with finite number of intersections, then the number of intersections of $\gamma_1$ and $\gamma_2$ is a decreasing function in $t$. Similarly, the number of intersections of $\gamma$ with any cone $C^a_b$ is also a decreasing function in $t$. 
			
			\item If $L_\gamma$ has a finite-time singularity, then it must occur at the origin. 
			
			\item The type I blow-up of any singularity is the cone $C^0_{\pi/2}$.
			
			\item The type II blow-up is a Lawlor neck asymptotic to the type I blow-up. The blow-up is independent of rescaling sequence.
			
			\item Any sequence of connected components as in the statement of Theorem \ref{Theorem-B-CP2} converges to a multiplicity 1 copy of the cone $C^0_{\pi/2}$.
		\end{enumerate}
	\end{prop}

	\begin{proof}
		See \cite{Evans_2022} for more detail on the following.
		\begin{enumerate}
			\item All 3 statements may be proven by variations on the same argument, which dates back to Angenent \cite{angenent_1991} applying a classical result of Sturm \cite{sturm_1836} on the zeroes of a uniformly parabolic PDE (see \cite[Proposition 1.2]{angenent_1991}). 
			
			\item Suppose $L_\gamma$ has a finite-time singularity away from the origin. Taking a type II blow-up at that point, we see that the origin is blown away to infinity, hence the $S^1$ symmetry becomes a translational symmetry. Therefore the type II blow-up is an eternal flow that splits as $\tau \times \bR$ for some curve $\tau$.
			
			Since the singularity must be type II, the second fundamental form $A$ of the type II blow-up has $|A| = 1$ at the space-time point $(0,0)$ by definition. So the geodesic curvature of $\tau$ is non-zero, and therefore the mean curvature of the type II blow-up is non-zero. However, application of Theorem B implies that we obtain a blow-up with constant Lagrangian angle, and hence the blow-up is minimal, a contradiction.
			
			\item By Theorem B, any connected component converges to a special Lagrangian cone. The $\bZ_2$ symmetry restricts the possible blow-ups to $C^0_{\pi/2}$ and $C^{\pi/4}_{\pi/2}$, with angles 0 or $\pi$, and $\pm \pi/2$ respectively.
			
			We then have to eliminate the possibility of $C^{\pi/4}_{\pi/2}$ occurring. We present the basic idea here, referring the reader to \cite{Evans_2022} for the complicated details.
			
			If a connected component $\sigma^i$ converging to  $C^{\pi/4}_{\pi/2}$ intersects the real (or equivalently imaginary axis), then $\sigma^i$ is $\bZ_2$ reflected across the real axis. But then the angle on one side of the axis is strictly determined by the angle on the other side, and it is relatively straightforward to show that such a connected component cannot converge in angle, violating Theorem B. So any connected component cannot intersect either the real or imaginary axis.
			
			Now that $\sigma^i$ is trapped in, for example, the positive quadrant, we aim to show that $\sigma^i$ converges to a Lawlor neck-type singularity. To do so, we first have to show no higher multiplicity can arise from a single connected component $\sigma^i$. 
			
			A long but somewhat standard density argument using Huisken's monotonicity formula yields a bound on the density on any annulus: Let $\sigma^i_k$ be any connected component of $\sigma^i \cap A(R,2R)$. Then we can show that
			\begin{equation}\label{eqn-density-2} \lim_{i\to \infty} \frac{ \mathcal H^1(\sigma^i_k \cap A(R,2R))}{R} = 2.\end{equation}
			Now the topology of the situation implies that a higher multiplicity cannot arise from a single connected component $\sigma^i$. 
			
			Hence the component $\sigma^i$ converges to a single density copy of $C^{\pi/4}_{\pi/2}$. But its mirror $\bar \sigma^i$ has the same convergence, so we have a double density copy of $C^{\pi/4}_{\pi/2}$. In this case however, we can guarantee that each component intersects the real axis and the imaginary axis within some small ball about the origin. This is the content of the Scale Lemma, which we prove immediately after this proof since it is important in the rest of the paper. 
			
			\item The proof is similar to that found in Wood \cite{wood_2020}.
			
			\item Let $\Sigma_j$ be a sequence of connected components converging to a multiplicity 2 copy of the cone $C^0_{\pi/2}$. Since $\Sigma_j$ are connected within a ball $B_R$ of small radius $R$ for sufficiently large $j$, and since they converge to a multiplicity 2 copy of the cone $C^0_{\pi/2}$, then within $B_R$, either $\Sigma_j$ intersects the positive real axis greater than once, or $\Sigma_j$ intersects both the real axis and the imaginary axis. Since they are connected within $B_R$ and are equivariant, we must have that they bound a disc $D \subset B_R$. But $D \subset B_R$ implies $\kappa \int_D \omega < \kappa \int_{B_R} \omega << 2\pi $, so $L$ is not monotone before the singular time, a contradiction.
		\end{enumerate}
	\end{proof}
	
	We conclude this section with an essentially important lemma, which enables us to define surgery at singularities. The same proof but with the cone $C^{\pi/4}_{\pi/2}$ completes the proof of the third part of Proposition \ref{Wood-lemmas}.
	
	\begin{lem}[Scale lemma]\label{lem-scale}
		Let $L_\gamma$ be a monotone equivariant Lagrangian mean curvature flow in $\bCP^2$ with a finite-time singularity at the origin at time $T < \infty$. Suppose the type I blow-up is the cone $C^0_{\pi/2}$ and the connected component converging to $C^0_{\pi/2}$ intersects the real axis.
		
		Then for any $R>0$, $\varepsilon_0>0$, $\delta>0$, there exists an $\varepsilon$ with $0<\varepsilon < \varepsilon_0$ and a time $t'$ with $T-\delta< t'<T$ such that $L_\gamma \cap B_R$ intersects $C^0_{\pi/2+2\varepsilon}$ at the time $t'$, where $B_R := B_R(0)$ is a ball of (Euclidean) radius $R$ at the origin.
	\end{lem}
	
	In essence, the scale lemma guarantees that singularity formation happens on an arbitrarily small scale. When we do surgery, this allows us to reduce the number of intersections with the cone $C^0_{\pi/2}$, thus controlling the total number of surgeries any flow can undergo. Currently, no such result exists for Lagrangian mean curvature flow in Euclidean space. Indeed, the proof we give relies heavily upon barriers that cannot exist in Euclidean space.
	
	\begin{proof}
		\begin{figure}
			\centering
			\includegraphics[scale=2]{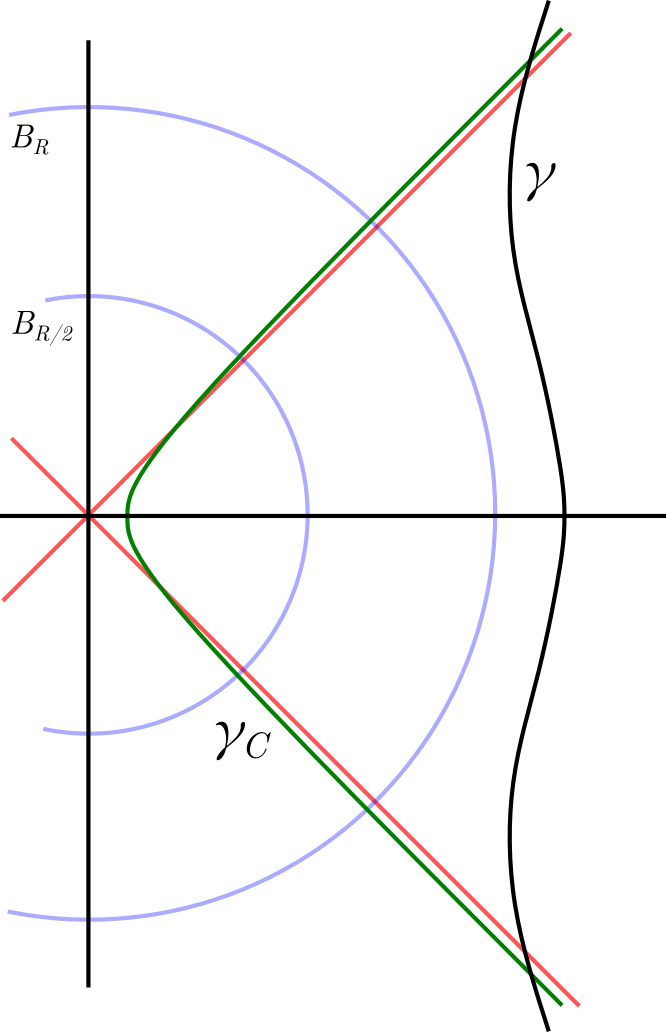}
			\caption{In order to form a singularity, $\gamma$ must eventually intersect $\gamma_C$ within $B_{R/2}$. But to do so, it must first intersect $\gamma_C$ within the annulus $A(R/2,R)$, and hence intersects a cone $C^0_{\pi/2 +2 \varepsilon}$ for $\varepsilon>0$.}\label{fig-scale-lemma}
		\end{figure}
		
		Suppose not. Then there exists $R>0$, $\varepsilon_0>0$, $\delta>0$ such that for any $\varepsilon$ with $0<\varepsilon < \varepsilon_0$ and $t'$ with $T- \delta < t'<T$, $L_\gamma \cap B_R$ does not intersect $C^0_{\pi/2 +2 \varepsilon}$. Since $T$ is the first singular time, we can, by taking a smaller $R>0$ if necessary, have that $L_\gamma \cap B_R$ is empty at time $T-\delta$. We may also freely assume $R<1$.
		
		Note that since $\gamma$ has a finite-time singularity at the origin, we have that there exists some $Q>0$ such that $\max_{p \in \gamma} r(p) < Q$ for all time $t \in [0,T)$, where $r(p)$ is the Euclidean distance to the origin.
		
		Consider the complete immersed minimal equivariant surfaces $L_{\gamma_C}$ constructed in Theorem \ref{thm-minimal-immersed-lags}, where $\gamma_C(s) = r_C(s)e^{is}$ has initial values $r(0) = r_1, r'(0) =0$. Since $L_{\gamma_C}$ is complete, $\gamma_C$ intersects $\gamma$ a finite number of times, non-increasing under the flow. By choosing $C>27$ sufficiently large, we can find a curve $\gamma_C$ such that 
		\begin{enumerate}
			\item $\gamma_C$ has maximum $r_2 > Q$.
			\item $r_C(s) < R/2$ for all $s \in [-\pi/4,\pi/4]$, see Proposition \ref{prop-cone-radius}. 
			\item The inner period $\psi^-_C$ (see Lemma \ref{lem-period}) satisfies $\psi^-_C < \pi/2 + 2 \varepsilon_0$.
		\end{enumerate}
		Consider the set
		\[\mathcal S = \left\{p \in \gamma \, : \, p = \gamma_C(s) \text{ for } s \in [0,\psi_C/2]\right\} ,\]
		where $\psi_C$ is the period of $\gamma_C$. Since $\gamma$ is monotone and both $\gamma$ and $\gamma_C$ are complete, this set is non-empty and finite for all time. Furthermore, if $p \in \mathcal S$, then $r(p) \in [R,Q]$ for all time by assumption, since otherwise we would have found a point $p \in L_\gamma \cap B_R$ intersecting $C^0_{\pi/2 +2 \varepsilon}$ for some $\varepsilon$ with $0<\varepsilon < \varepsilon_0$. Since the number of intersections between $\gamma$ and $\gamma_C$ is non-increasing, $L_\gamma \cap L_{\gamma_C} \cap B_R$ is empty for all time. 
		
		However, since the connected component giving $C^0_{\pi/2}$ in the blow-up contains the real axis, we have that the intersection with the real axis converges to 0 since otherwise the blow-up would have to contain a line $C^0_a$ with $a< \pi/4$. In particular, there must exist some time $\tilde t$ with $T-\delta < \tilde t < T$ when $L_\gamma \cap L_{\gamma_C }\cap B_R$ is non-empty, a contradiction. 
	\end{proof}
	
	\begin{rem}\label{rem-lawlor-vs-cone}
		The final step of the above proof can be simplified using the assumption that the Lawlor neck is the type II blow-up. 
	\end{rem}

	\subsection{Clifford tori}\label{sec-graph-Cliff}
	
	A natural condition to impose on solutions of the equivariant flow (\ref{eqn-equiv-mcf}) is that $\gamma$ is graphical over the minimal equivariant Clifford torus $L_1$. We have already studied minimal solutions of (\ref{eqn-equiv-mcf}): the only embedded minimal solution is $L_1$. Thus, by Proposition \ref{Prop-infinite-sing}, if we can prove that $\gamma$ has long-time existence, then we obtain convergence to $L_1$ in infinite time. 
	
	\begin{rem}	 	
		Recall the fibration $\{L_\alpha\}$ by Clifford-type tori given by the moment map 
		\[\mu([x:y:z]) = \frac{1}{|x|^2 +|y|^2 +|z|^2}\left( |x|^2, |y|^2\right).\]
		The equivariant fibres are
		\[L_r = \{L_{re^{i\phi}}:r >0\}\]
		and from here on we denote by $\Omega$ the holomorphic volume form relative to $L_r$. Then for a Lagrangian $L$, we defined the Lagrangian angle $\theta$ relative to $\Omega$ by
		\[\Omega_L = e^{i\theta} \operatorname{vol}_L.\]
		Recall that in the Calabi--Yau case, if $\theta$ can be chosen to be a real-valued function, we call $L$ zero-Maslov with respect to $\Omega$. Furthermore, we call $L$ almost-calibrated with respect to $\Omega$ if there exists some $\delta >0$ such that
		\[ \cos \theta  > \delta>0.\] 
		Note that $\gamma$ is graphical over $L_1$ if and only if $\gamma$ is almost-calibrated with respect to $\Omega$.
		
		However, unlike in the Calabi--Yau case, $\theta$ defined in this way satisfies the evolution equation
		\[\frac{\partial}{\partial t} \theta = \Delta \theta + d^\dagger \alpha,\]
		and hence the parabolic maximum principle does not imply that $\cos \theta$ is increasing in time. Indeed, consider the following setup: let $\gamma$ be a small ellipse with eccentricity $0<e<1$ centred on the origin. Then $\gamma$ has a finite-time singularity at the origin. Furthermore, if $e$ is sufficiently close to 1, the singularity is type II and has type II blow-up a Lawlor neck. In particular, there is no constant $\delta>0$ such that $\cos \theta > \delta$ for all time. So almost-calibrated is not preserved in general. 
		
		Furthermore, even in the case where $L_\gamma$ is monotone, we should not expect almost-calibrated to be preserved locally. Indeed, we construct an example later in the paper where a non-graphical Clifford torus forms a finite-time singularity. However the construction seems to indicate that the almost-calibrated condition breaks locally in this case.
		
		These examples illustrates two ideas. Firstly, we should consider $\gamma$ as graphical rather than almost-calibrated. Secondly, we should only expect graphical to be preserved in the case that $\gamma$ gives a monotone torus $L_\gamma$. This concludes the remark.
	\end{rem}
	
		\begin{prop}\label{Prop-AC-Cliff}
		Let $L_\gamma$ be a monotone equivariant Clifford torus, graphical over the minimal equivariant Clifford torus $L_1$. Then under mean curvature flow, $L_\gamma$ exists for all time and converges to $L_1$ in infinite time.
	\end{prop}
	
	\begin{proof}
		As mentioned above, it suffices to show that we have long-time existence. 	First, note that the graphical condition is preserved up to any potential singular time since the number of intersections of $\gamma$ with any cone $C^a_b$ is decreasing in time by Proposition \ref{Wood-lemmas}. 
		
		Suppose for a contradiction that we have a finite time singularity at time $T$. Any finite-time singularity must occur at the origin by Proposition \ref{Wood-lemmas}, and must have blow-up given by $C^0_{\pi/2}$. 
		
		For any graphical $\gamma$ and $\varepsilon>0$, the cone $C^0_{\pi/2-2\varepsilon}$ intersects $\gamma$ 4 times, dividing the Maslov 4 disc into 4 triangles. Denote the triangles intersecting the positive and negative real axes by $P^+_\varepsilon$ and $P^-_\varepsilon$, and the triangles intersecting the positive and negative imaginary axes by $Q^+_\varepsilon$ and $Q^-_\varepsilon$. Note that since $L_\gamma$ is monotone, 
		\[\int_{P^+_\varepsilon + P^-_\varepsilon + Q^+_\varepsilon + Q^-_\varepsilon} \omega = 4\pi/6 = 2\pi/3.\]
		Suppose the type II blow-up of a connected component is a Lawlor neck intersecting the real axis (this assumption is reasonable and simplifies the proof, but can be removed, see Remark \ref{rem-lawlor-vs-cone}). Then for any $\varepsilon>0$, we have that 
		\[\int_{P^+_\varepsilon} \omega \to 0\]
		as $t \to T$, where $P_\varepsilon$ is the $J$-holomorphic triangle bounded by $\gamma$ and $C^0_{\pi/2-2\varepsilon}$, intersecting the positive real axis and with a vertex at 0. But the total area contained outside the cone is bounded above by $\pi/2 + 2\varepsilon$, so 
		\[\int_{Q^+_\varepsilon + Q^-_\varepsilon} \omega < \pi/2 + 2\varepsilon.\]
		Let $\varepsilon = \pi/48$. We can find a time $t$ close to $T$ such that
		\[\int_{P^+_\varepsilon + P^-_\varepsilon} \omega < \pi/24.\]
		Then at $t$,
		\[\int_{P^+_\varepsilon + P^-_\varepsilon + Q^+_\varepsilon + Q^-_\varepsilon} \omega < \pi/2 + \pi/12 =7\pi/12 < 2\pi/3,\]
		which contradicts $L_\gamma$ being monotone.				
	\end{proof}
	
	\begin{rem}
		As in Remark \ref{rem-lawlor-vs-cone}, the assumption that the type II blow-up is a Lawlor neck simplifies the proof, but is not necessary.
	\end{rem}
	
	\subsection{Chekanov tori}\label{sec-Chek}

	In the following, we will analyse the behaviour of equivariant Chekanov tori under mean curvature flow. Note first of all that any equivariant Chekanov torus does not intersect either the imaginary or real axis. Without loss of generality assume the former. Then there exists a cone $C^0_\psi$ of maximal opening angle $\psi$ such that $C^0_\psi \cap L_\gamma$ is non-empty. Since $L_\gamma$ is monotone and the area inside the cone $C^0_\psi$ is $\psi/2$, we must have that $\psi > 2\pi/3$.
	
	We begin by proving there are no minimal equivariant Chekanov tori. This is interesting in its own right, and the method of proof suggests it may generalise to the non-equivariant case.
	
	\begin{prop}\label{prop-no-min-Chek}
		There is no minimal equivariant Chekanov torus.
	\end{prop}	
	
	\begin{proof}
		The result is immediate by the classification of equivariant tori in Section \ref{sec-minimal}. However, we present a different proof since we believe the idea of the proof may be more widely applicable.
		
		Let $L_\gamma$ be an equivariant Chekanov torus. Without loss of generality, we are free to restrict to the subclass of Chekanov tori $L_\gamma$ for which $\gamma$ does not intersect the imaginary axis in $\bC$.
		
		Since $\gamma$ does not intersect the imaginary axis and $L_\gamma$ is equivariant, there is some maximal angle $\psi$ such that $C^0_{\psi}$ intersects $\gamma$. Since $\gamma$ does not pass through the origin and does not intersect the imaginary axis, $\psi = \pi - \delta$ for some $\delta >0$. Denote the first points of intersection of $L_\gamma$ and $C^0_{\psi}$ by $p^+$ and $p^-$, i.e. if $p \in C^0_\psi \cap L_\gamma$, then $r(p^\pm) \leq r(p)$.
		
		\begin{figure}
			\centering
			\includegraphics[scale=2]{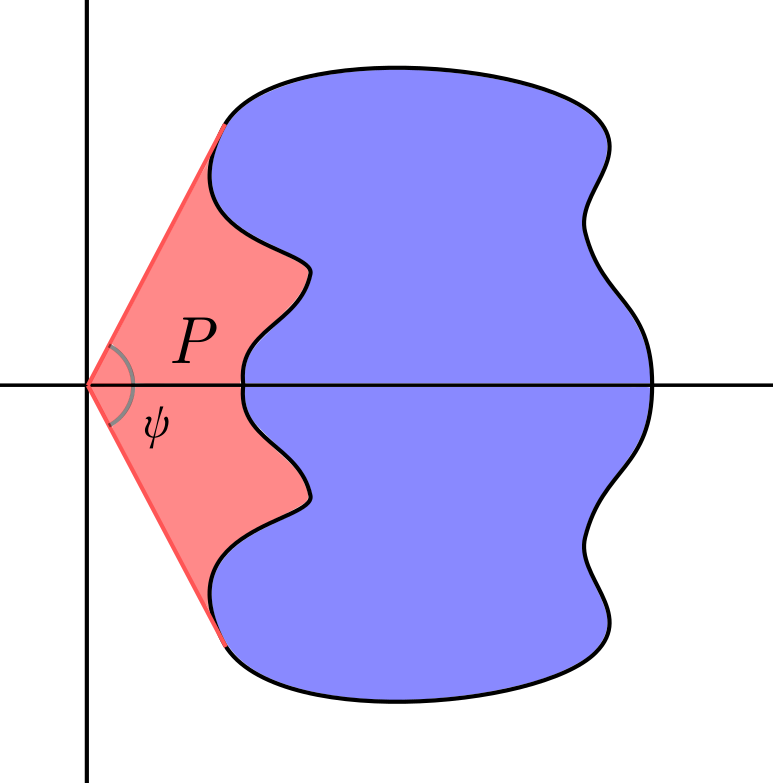}
			\caption{The triangle $P$ collapses in finite time for a Chekanov torus.}\label{fig-chek-sing}
		\end{figure}
		
		Consider the $J$-holomorphic triangle $P$ with boundary on $l_{\psi/2}, l_{-\psi/2}$ and $\gamma$ with one vertex at the origin and the other two vertices at $p^+$ and $p^-$. As in Example \ref{ex-triangle}, 
		\[\tilde \mu(P) = - \frac{1}{\pi} \left(\pi-2\psi\right),\]
		and hence by (\ref{eqn-CG-gen}) we have that
		\[-\int_{\gamma_0} H = \kappa \int_P\omega + \pi - 2\psi,\]
		where $\gamma_0$ is the arc of $\gamma$ running from $p^+$ to $p^-$. But
		\[ \int_P \omega < \frac{\psi}{2} - \frac{\pi}{3}\]
		since the area of $P$ is bounded by the difference between the total area contained in the cone and the area of the Maslov 2 disc bounded by $\gamma$. So
		\[-\int_{\gamma_0} H < 3\psi -2\pi + \pi - 2\psi = \psi -\pi =-\delta<0.\] 
		Hence $L_\gamma$ is not minimal.
	\end{proof}
	
	\begin{cor}\label{Prop-Chek-sing}
		Let $L_\gamma$ be an equivariant Chekanov torus in $\bCP^2$. Then under mean curvature flow, $L_\gamma$ has a finite-time singularity at the origin $[0:0:1]$.
	\end{cor}
	
	\begin{proof}
		By Proposition \ref{prop-no-min-Chek}, there is no minimal equivariant representative in the Hamiltonian isotopy class of $L_\gamma$. Proposition \ref{Prop-infinite-sing} therefore implies that we have a finite-time singularity, which must occur at the origin by Proposition \ref{Wood-lemmas}.
	\end{proof}
	
	We also proffer an alternative proof using the evolution equation derived in Lemma \ref{lem-tangential-triangle}.
	
	\begin{proof}
		Suppose there is no finite-time singularity. We are in the situation of Lemma \ref{lem-tangential-triangle}, noting that the maximum opening angle must be a smooth function of $t$ for all $t$ sufficiently large. By the argument above, the right hand side of (\ref{eqn-ev}) is less than $-\delta$ for some $\delta>0$. But then there is only a finite period of time when $P$ has positive area, a contradiction. 
	\end{proof}	

	\subsection{Neck-to-neck surgery and the proof of the main theorem}\label{sec-surgery}
	
	In this section we analyse the behaviour of Lagrangian tori at the singular time. The equivariant condition necessarily (and intentionally) restricts us to two Hamiltonian isotopy classes of Lagrangian tori: the Clifford and Chekanov tori. We have shown in Proposition \ref{Prop-Chek-sing} that an equivariant Chekanov torus achieves a type II singularity at the origin with type II blow-up given by a Lawlor neck. Resolving the Lawlor neck singularity by neck-to-neck surgery gives a Clifford torus. We have also shown that almost-calibrated  Clifford tori have long-time existence and convergence to the equivariant minimal Clifford torus. Our dream is that under Lagrangian mean curvature flow with surgeries, exotic tori in $\bCP^2$ flow towards a minimal Clifford torus in infinite time, so we are led to ask the following question in our symmetric case:
	
	\begin{ques}
		Does an equivariant Chekanov torus flow to a minimal Clifford torus after neck-to-neck surgery?
	\end{ques}
	
	The answer to this question is yes, with the caveat that the number of neck-to-neck surgeries may be greater than one. 
	
	In fact, we will prove the following:
	
	\begin{thm}\label{Thm-main-conv}
		Let $L_\gamma$ be a equivariant Clifford or Chekanov torus in $\bCP^2$. Then, after a finite number of neck-to-neck surgeries, $L_\gamma$ converges to the unique equivariant minimal Clifford torus $L_1$ in infinite time.
	\end{thm}
	
	\begin{defn}
		Let $L_\gamma$ be an equivariant Lagrangian mean curvature flow with a finite-time singularity at $[0:0:1]$ at time $T < \infty$. Suppose the type I blow-up is the cone $C^0_{\pi/2}$ and the type II blow-up is (independent of rescaling) the Lawlor neck asympototic to $C^0_{\pi/2}$ intersecting the real axis. 
		For any $r>0$, we can find $\varepsilon>0$ and a least time $t'<T$ such that $L_\gamma \cap B_r$ intersects $C^0_{\pi/2+2\varepsilon}$ as in Lemma \ref{lem-scale}. We define a new curve $\zeta$ in $\bC$ which will give a Lagrangian $L_\zeta$, which we will call the scale $r$ surgery of $L_\gamma$. 
		
		Let $p^{\pm} = r' e^{\pm i(\pi/4 +\varepsilon)}$ be the smallest radius points of intersection of $L_\gamma \cap B_r$ and $C^0_{\pi/2+2\varepsilon}$. Similar to the proof of Lemma \ref{lem-scale}, we can find a curve segment $\zeta'$ (intersecting the imaginary axis this time), smoothly tangent to $C^0_{\pi/2+\varepsilon}$ at points $q^\pm$, $-q^{\pm}$ at radius $r''<r'$. Define $\zeta$ to be the union of
		$\zeta'$, $\gamma \cap A(r',\infty)$ and a smooth curve interpolating between $p^\pm$ and $q^\pm$.
		
		Rescale radially and perform Moser's trick as in Vianna's construction to obtain from $L_\zeta$ a monotone surgery of $L_\gamma$.	We call this procedure neck-to-neck surgery.		
	\end{defn}
	
	\begin{rem}{$\ $}\\[-3ex]
		\begin{enumerate}
			\item On the level of Lagrangians, this construction is not canonical. Since we need to use Moser's trick to obtain a monotone torus, monotone surgery is never going to be canonical unless you can flow directly through the singularity. In this case however, there is a canonical way to perform Moser's trick since the equivariance means you can just rescale radially until you obtain a monotone torus. This is a quirk of the equivariance and cannot be expected in general. 
			
			\item The surgery procedure does not require that the type I blow-up is multiplicity 1 (even though we conjecture that all type I blow-ups in this situation are multiplicity 1). In the case that the multiplicity is higher than 1, the closest intersection point with the cone $C^0_{\pi/2+2\varepsilon}$ is continuous in time for times $t$ sufficiently close to the singular time $T$. Hence there is no ambiguity about the neck to be cut and rotated in any case.
			
			\item The surgery is canonical from a symplectic point of view since we always land in the same Hamiltonian isotopy class.
			
			\item The surgery procedure is designed to reduce the number of intersections of $L_\gamma$ with $C^0_{\pi/2}$. Since we rescale radially, applying Moser's trick does not alter the number of intersections.			
		\end{enumerate}
	\end{rem}
	
	\begin{proof}[Proof of Theorem \ref{Thm-main-conv}]
		As in the proof of Proposition \ref{Prop-Chek-sing}, we restrict to the case where $L_\gamma$ intersects the imaginary axis either twice (in the case of a Clifford torus) or not at all (in the case of the Chekanov torus). Then by Proposition \ref{Wood-lemmas}, equivariant tori can only achieve finite-time singularities at the origin $[0:0:1]$ with type I blow-up given by $C^0_{\pi/2}$, and type II blow-up given by a Lawlor neck asymptotic to $C^0_{\pi/2}$.
		
		Since $L_\gamma$ is compact, it has a finite number of intersections with $C^0_{\pi/2}$. By Proposition \ref{Wood-lemmas}, the number of intersections is a decreasing function of time under mean curvature flow, and by definition neck-to-neck surgery decreases the number of intersections with $C^0_{\pi/2}$. 
		
		If $L_\gamma$ is a Chekanov torus, Proposition \ref{Prop-Chek-sing} guarantees a finite-time singularity, at which point $L_\gamma$ becomes a Clifford torus after surgery. Similar to the proof of Proposition \ref{Prop-Chek-sing}, if $L_\gamma$ is a Clifford torus intersecting $C^0_{\pi/2}$ greater than 4 times, then $L_\gamma$ has inflection points and hence cannot be minimal. So $L_\gamma$ either has a finite-time singularity, or after a finite time has only 4 intersections with $C^0_{\pi/2}$.
		
		Since the number of intersections is strictly decreasing after surgery and both the above cases end in either surgery or a reduction of the number of intersections, after a finite time and a finite number of surgeries, we have the minimum number of intersections. It has already been shown in Proposition \ref{Prop-AC-Cliff} that if $L_\gamma$ has 4 intersections with $C^0_{\pi/2}$, then $L_\gamma$ exists for all time and converges to $L_1$. The result follows.
	\end{proof}	

	\begin{rem}
		As in the secondary proof of Proposition \ref{Prop-Chek-sing} we could work with the evolution equations for $\int_P \omega$ directly, rather than just showing there are no minimal objects and applying Proposition \ref{Prop-infinite-sing}. 
	\end{rem}
	
	\subsection{A Clifford torus with two singularities}\label{sec-2sings}
	
	Further to the result of the previous section, we also give a proof of the following existence result:
	
	\begin{prop}\label{Prop-exist-n-sing}
		Let $n\geq 0$ be an integer. Then 
		\begin{enumerate}
			\item There exists a Clifford torus that undergoes exactly $2n$ neck-to-neck surgeries before converging to a minimal Clifford torus.
			
			\item There exists a Chekanov torus that undergoes exactly $2n+1$ neck-to-neck surgeries before converging to a minimal Clifford torus.
		\end{enumerate}
	\end{prop}
	
	We give an explicit construction for the $n=1$ case. The construction is somewhat technical but the idea is fairly simple: construct a monotone Clifford torus that has curvature sufficiently high in a neighbourhood of the origin, and use a barrier to stop $\gamma$ from crossing the cone $C^0_{2\pi/3}$ for long enough that a singularity is inevitable. The constants chosen in the course of the proof are of no particular significance. 
	
	First, we prove a small lemma concerning the type I singular time of non-monotone Chekanov tori.
	
	\begin{lem}\label{lem-shrinking-chek}
		Let $L_\zeta$ be a non-monotone equivariant torus bounding a Maslov 2 disc $D$ of area $A = \int_D \omega < \pi/3$. If $\zeta$ has $r > R(A)  = \sqrt{A(\pi-2A)^{-1}} $ everywhere, then $L_\zeta$ has a type I singularity away from the origin at time 
		\[T =  \frac{1}{6} \log\left( \frac{\pi}{3A-\pi}\right).\]
	\end{lem}

	\begin{proof}
		We have that 
		\[	\frac{d}{d t} \int_D \omega = \kappa \int_D \omega -2 \pi ,\]
		so denoting $f(t) = \int_D \omega$, we have 
		\[f(t) = \left(A - \frac{\pi}{3}\right) e^{6t} + \frac{\pi}{3}.\]
		Hence the final existence time $T$ of $L_\gamma$ satisfies
		\[T \leq \frac{1}{6} \log\left( \frac{\pi}{3A-\pi}\right),\]
		with equality if the singularity is type I. 
		
		Now consider $L_{R(A)}$ given by $\gamma_{R(A)}(s) = R(A) e^{is}$. Equation (\ref{eqn-area-disc}) reveals that $L_{R(A)}$ bounds a Maslov 4 disc of area $B =  (2 \pi R(A)^2)(1+2R(A)^2)^{-1}$, so $B>2A$. Furthermore, a similar calculation to above gives the final existence time $T'$ of $L_{R(A)}$ as
		\[T' = \frac{1}{6}\log\left(\frac{2\pi}{3B-2\pi}\right).\]
		Note that if $L_\zeta$ has a type II singularity, it must be at the origin and since $L_{R(A)}$ is a barrier to $L_\zeta$, it must occur after $T'$. But $B>2A$ implies that $T' > T$, and the result follows.  
	\end{proof}
	
	\begin{proof}[Proof of Proposition \ref{Prop-exist-n-sing}]
		We construct the $n=1$ case, i.e. an equivariant Clifford torus $L_\gamma$ with a finite-time singularity. The cases $n>1$ follow an iterated version of the $n=1$ case, and the Chekanov case follows automatically from the Clifford case.
		
		\begin{figure}[h]
			\centering
			\includegraphics[scale=2]{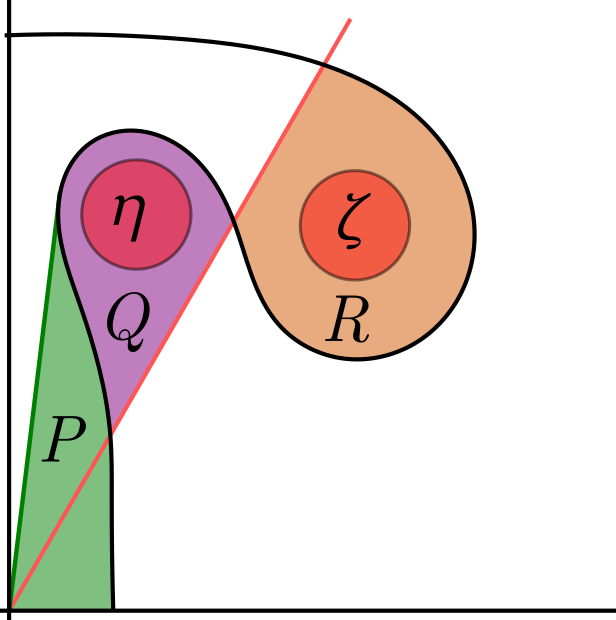}
			\caption{The construction of a Clifford torus with a finite-time singularity. Choosing the area of $P$ sufficiently small and the areas of $Q$ and $R$ sufficiently large guarantees a finite-time singularity.}\label{fig-cliff2sings}
		\end{figure}
		
		See Figure \ref{fig-cliff2sings} for the construction we now describe. Since $\gamma$ is equivariant, it suffices to describe the construction only in the positive real quadrant, i.e. the region 
		\[Z = \{re^{i\phi} \in \bC: \phi \in [0,\pi/2]\}.\]
		
		Let $\gamma$ be a equivariant curve such that $L_\gamma$ is a Clifford-type torus and $\gamma$ intersects the cone $C^0_{2\pi/3}$ 3 times in $Z$. Suppose $\gamma$ has the minimum of 2 inflection points, i.e. points with $\langle \gamma,\nu\rangle =0$. Note that then there are two biangles $Q$ and $R$ bounded by $\gamma$ and $C^0_{2\pi/3}$. Furthermore, there is a triangle $P$ formed by $\gamma$ and the cone $C^0_{\psi}$ where $\psi$ is the widest opening angle such that $C^0_\psi$ intersects $\gamma$ exactly 2 times in $Z$. 		
		
		We can choose $\gamma$ such that we can find two Euclidean circles $\zeta \subset Q$ and $\eta \subset R$ each bounding discs with area $\pi/18$. Furthermore, we can choose $\gamma$ such that $\zeta$ and $\eta$ both have $r > R(\pi/18)$ as in the requirements of Lemma \ref{lem-shrinking-chek}.
		
		Furthermore, we can choose $\gamma$ such that the triangle $P$ has area $\pi/216$. It is clear we can make these choices whilst also choosing $\gamma$ such that $L_\gamma$ is monotone.  
		
		The non-monotone Chekanov tori $L_\zeta$ and $L_\eta$ give lower bounds for the final time $\int_Q \omega, \int_R \omega >0$ via direct application of Lemma \ref{lem-shrinking-chek}. We have that that $\int_Q \omega>0 $ and $\int_R \omega >0$ (and hence $\psi > 2\pi/3$) for all times $t<T$ with
		\[T := \frac{1}{6} \log\left(\frac{6}{5}\right).\]
		
		Suppose for a contradiction that the flow exists on the interval $[0,T]$. Then the triangle $P$ exists for that period also, and the evolution of the triangle $P$ is as in Lemma \ref{lem-tangential-triangle}. We have that 
		\[\frac{d}{d t} \int_P \omega  \leq \kappa \int_P \omega + \left(\pi - 2\psi\right) \leq  \kappa \int_P \omega - \frac{\pi}{3},\]
		since $\psi$ is decreasing under the flow. Denoting $u(t) = \kappa \int_P \omega - \pi/3$, we see that $u(t)$ satisfies the differential inequality
		\[\frac{d}{d t} u(t) \leq \kappa u(t),\]
		to which we can apply Gr\"onwall's inequality. Since $\kappa = 6$, we deduce that $u(t)$ satisfies
		\[u(t) \leq u(0) e^{6t},\]
		and hence 
		\[\int_P \omega \leq \left(\frac{\pi}{216} - \frac{\pi}{18}\right) e^{6t} + \frac{\pi}{18} = \left(1-\frac{11}{12}e^{6t} \right)\frac{\pi}{18}\]
		Hence the triangle $P$ has a maximum existence time of 
		\[T' : = \frac{1}{6} \log \left(\frac{12}{11}\right),\]
		which is strictly less than $T$, a contradiction. Hence a finite-time singularity occurs in the period $[0,T]$. 
	\end{proof}
	
	\appendix
	\section{Type II singularities and blow-ups in Lagrangian mean curvature flow}\label{appendix}
	
	Singularities in mean curvature flow are classified into two types based on the rate of blow-up of the second fundamental form $A$. For a singularity at time $T$, if 
	\[\sup_{L_t} |A|^2 \leq  C(T-t)^{-1},\]
	for some constant $C$, we call the singularity type I, and if no such bound exists, we call it type II. The primary reason for this distinction is the following: for $\lambda >0$, $\tilde x = \lambda(x-x_0)$, $\tilde t = \lambda^2(t-t_0)$,
	\[\tilde F^\lambda_{\tilde t} 	:= \lambda \left( F_{t_0 +\lambda^{-2} \tilde t} - x_0\right)\]
	is a mean curvature flow, called a parabolic rescaling. Using Huisken's monotonicity formula, one can show that any sequence of parabolic rescalings $F^{\lambda_i}_{\tilde t}$ with $\lambda_i \to \infty$ at a type I singularity $(x_0,t_0) = (x,T)$ of the flow converges subsequentially to a smooth limiting flow $F_{\tilde t}^\infty$, called a type I blow-up (possibly not unique), with the property that 
	\[\vec H = \frac{x^\perp}{2\tilde t}.\]
	Solitons of mean curvature flow of this type are called \textit{self-shrinkers} since they flow by homotheties. If the singularity is instead type II, one still can find a weak limit to parabolic rescalings, though now the limiting flow is a \textit{Brakke flow} \cite{Brakke1978}: a flow of rectifiable varifolds rather than smooth manifolds. We also call this limit a type I blow-up, even though the singularity is type II.
	
	The most fundamental results on singularities in Lagrangian mean curvature flow are the compactness results of Neves, originally established for zero-Maslov Lagrangians in \cite{Neves2007} but later extended to monotone Lagrangians in \cite{Neves2010}. We present them in the latter form since it is more applicable to the subject matter of this paper. Compact monotone Lagrangians in $\bC^n$ have a maximal existence time strictly controlled by the monotone constant. One can always normalise by homotheties of the ambient space so that this maximal time of existence is $1/2$. Neves' theorems concern singularities happening before this time.
	
	\begin{thm}[Neves' Theorem A]\label{thm-A-Nev-mono}
		Let $L$ be a normalised monotone Lagrangian in $\bC^n$ developing a singularity at $T<1/2$. For any sequence of rescaled flows $L^j_s$ at the singularity with Lagrangian angles $\theta^j_s$, there exists a finite set of angles $\{\bar \theta_1, \dots, \bar \theta_N\}$ and special Lagrangian cones $L_1,\dots,L_N$ such that after passing to a subsequence we have that for any smooth test function $\phi$ with compact support, every $f \in C^2(S^1)$ and $s<0$
		\[ \lim_{j \to \infty} \int_{L^j_s} f(\exp(i\theta^j_s)) \phi \, d\mathcal H^n = \sum_{k=1}^N m_k f(\exp(i \bar \theta_k)) \mu_k(\phi),\]
		where $\mu_k$ and $m_k$ are the Radon measure of the support of $L_k$ and its multiplicity respectively.
		
		Furthermore the set of angles is independent of the sequence of rescalings.
	\end{thm}
	
	Theorem B applies for monotone Lagrangians in $\bC^2$. 
	
	\begin{thm}[Neves' Theorem B]\label{thm-B-Nev-mono}
		Let $L$ be a normalised monotone Lagrangian in $\bC^2$ developing a singularity at $T<1/2$. For any sequence of rescaled flows $L^j_s$ at the singularity  with Lagrangian angles $\theta^j_s$, and for any sequence of connected components $\Sigma_j$ of $L^j_s \cap B_{4R}(0)$ intersecting $B_R(0)$, there exists a unique angle $\bar \theta$ and special Lagrangian cone $\Sigma$ such that after passing to a subsequence we have that for any smooth test function $\phi$ on $B_{2R}(0)$ with compact support, every $f \in C^2(S^1)$ and $s<0$
		\[\lim_{j\to \infty} \int_{\Sigma_j} f(\exp(i\theta^j_s) \phi \, d\mathcal H^n =  m f(\exp(i \bar \theta)) \mu(\phi),\]
		where $\mu$ and $m$ are the Radon measure of the support of $\Sigma$ and its multiplicity respectively.
	\end{thm}
	
	Heuristically, these theorems give the type I blow-up models of type II singularities of Lagrangian mean curvature flow as (unions of) special Lagrangian cones. Consider the $n=2$ case: by considering the hyper-K\"ahler rotation, one see that the only special Lagrangian cones are unions of special Lagrangian planes with equal Lagrangian angle. Assuming all planes are multiplicity 1, there is then only one blow-up model up to rotation, a union of two transversely intersecting special Lagrangian planes with the same Lagrangian angle.
	
	We can further characterise singular behaviour by a procedure called the type II blow-up. The precise details of this procedure are not important to this paper, but we sketch the general principle. We refer the reader to Mantegazza \cite{mantegazza_2013} for additional details. Alternatively, the procedure is described in depth in \cite{wood_2020}, where the first examples of Lawlor necks as type II blow-ups were found. Instead of blowing up at a parabolic rate and at a fixed point in time and space, we blow up at a sequence of space-time points $(x_i,t_i)$ maximising the second fundamental form $|A|$ on the interval $[0,T-1/i]$, at a rate dictated by the second fundamental form $|A|$. Thus we guarantee convergence locally smoothly to an eternal mean curvature flow, i.e. a mean curvature flow existing for all times $t \in (-\infty,\infty)$ (as opposed to the self-shrinkers found by the type I procedure, which are ancient but not eternal). 
	
	Note that the type II blow-up is not unique and doesn't a priori have to satisfy the same asymptotics as the type I blow-up. There are few results on type II blow-ups for Lagrangian mean curvature flow so far, but the most important appears in the work of Wood \cite{Wood2019}, where he shows that almost-calibrated Lagrangian cylinders with prescribed asymptotic behaviour achieve type II singularities in finite time, and the type II blow-up is given by a special Lagrangian called a Lawlor neck asymptotic to the type I blow-up. 
	
	\bibliography{Symp}

\begin{bibdiv}
\begin{biblist}

\bib{abresch_langer_1986}{article}{
      author={Abresch, U.},
      author={Langer, J.},
       title={The normalized curve shortening flow and homothetic solutions},
        date={1986},
     journal={Journal of Differential Geometry},
      volume={23},
      number={2},
}

\bib{alston_amorim_2011}{article}{
      author={Alston, G.},
      author={Amorim, L.},
       title={Floer cohomology of torus fibers and real {L}agrangians in {F}ano
  toric manifolds},
        date={2011},
     journal={International Mathematics Research Notices},
}

\bib{angenent_1991}{article}{
      author={Angenent, S.},
       title={Parabolic equations for curves on surfaces: Part {II}.
  {I}ntersections, blow-up and generalized solutions},
        date={1991},
     journal={The Annals of Mathematics},
      volume={133},
      number={1},
       pages={171},
}

\bib{Auroux_2007}{article}{
      author={Auroux, D.},
       title={Mirror symmetry and {T}-duality in the complement of an
  anticanonical divisor},
        date={2007},
     journal={Journal of Gökova Geometry Topology},
      volume={1},
       pages={51\ndash –91},
}

\bib{biran_cornea_2009}{article}{
      author={Biran, P.},
      author={Cornea, O.},
       title={Rigidity and uniruling for {L}agrangian submanifolds},
        date={2009},
     journal={Geometry \& Topology},
      volume={13},
      number={5},
       pages={2881–\ndash 2989},
}

\bib{Brakke1978}{book}{
      author={Brakke, K.~M.},
       title={The motion of a surface by its mean curvature},
   publisher={Princeton University Press},
        date={1978},
}

\bib{castro_lerma_miquel_2018}{article}{
      author={Castro, I.},
      author={Lerma, A.~M.},
      author={Miquel, V.},
       title={Evolution by mean curvature flow of {L}agrangian spherical
  surfaces in complex {E}uclidean plane},
        date={2018},
     journal={Journal of Mathematical Analysis and Applications},
      volume={462},
      number={1},
       pages={637–\ndash 647},
}

\bib{chen_he_2010}{article}{
      author={Chen, J.},
      author={He, W.},
       title={A note on singular time of mean curvature flow},
        date={2010},
     journal={Mathematische Zeitschrift},
      volume={266},
       pages={921–\ndash 931},
}

\bib{Cieliebak2004}{article}{
      author={Cieliebak, K.},
      author={Goldstein, E.},
       title={A note on mean curvature, {Maslov} class and symplectic area of
  {Lagrangian} immersions},
        date={2004},
     journal={Journal of Symplectic Geometry},
      volume={2},
      number={2},
       pages={261–\ndash 266},
}

\bib{entov_polterovich_2009}{article}{
      author={Entov, M.},
      author={Polterovich, L.},
       title={Rigid subsets of symplectic manifolds},
        date={2009},
     journal={Compositio Mathematica},
      volume={145},
      number={03},
       pages={773\ndash –826},
}

\bib{Evans_2022}{thesis}{
      author={Evans, C.~G.},
       title={{L}agrangian mean curvature flow in the complex projective
  plane},
        type={Ph.D. Thesis},
organization={University College London},
        date={2022},
}

\bib{Evans2018}{article}{
      author={Evans, C.~G.},
      author={Lotay, J.~D.},
      author={Schulze, F.},
       title={Remarks on the self-shrinking {C}lifford torus},
        date={2020},
     journal={Journal f\"ur die reine und angewandte Mathematik (Crelle's
  Journal)},
      number={765},
       pages={139\ndash –170},
}

\bib{grayson_1989}{article}{
      author={Grayson, M.~A.},
       title={Shortening embedded curves},
        date={1989},
     journal={The Annals of Mathematics},
      volume={129},
      number={1},
       pages={71},
}

\bib{Groh2007}{article}{
      author={Groh, K.},
      author={Schwarz, M.},
      author={Smoczyk, K.},
      author={Zehmisch, K.},
       title={Mean curvature flow of monotone {L}agrangian submanifolds},
        date={2007},
     journal={Mathematische Zeitschrift},
      volume={257},
      number={2},
       pages={295\ndash 327},
}

\bib{Huisken2009}{article}{
      author={Huisken, G.},
      author={Sinestrari, C.},
       title={Mean curvature flow with surgeries of two--convex hypersurfaces},
        date={2009},
        ISSN={1432-1297},
     journal={Inventiones mathematicae},
      volume={175},
      number={1},
       pages={137\ndash 221},
         url={https://doi.org/10.1007/s00222-008-0148-4},
}

\bib{joyce_2015}{article}{
      author={Joyce, D.},
       title={Conjectures on {B}ridgeland stability for {F}ukaya categories of
  {C}alabi-–{Y}au manifolds, special {L}agrangians, and {L}agrangian mean
  curvature flow},
        date={2015},
     journal={EMS Surveys in Mathematical Sciences},
      volume={2},
      number={1},
       pages={1–\ndash 62},
}

\bib{Lotay2018}{article}{
      author={Lotay, J.~D.},
      author={Pacini, T.},
       title={From minimal {Lagrangian} to {J}-minimal submanifolds:
  persistence and uniqueness},
        date={2018},
     journal={Bollettino dell Unione Matematica Italiana},
}

\bib{mantegazza_2013}{book}{
      author={Mantegazza, C.},
       title={Lecture notes on mean curvature flow},
   publisher={Springer Basel},
        date={2013},
}

\bib{Neves2007}{article}{
      author={Neves, A.},
       title={Singularities of {Lagrangian} mean curvature flow: zero-{Maslov}
  class case},
        date={2007},
     journal={Inventiones mathematicae},
      volume={168},
      number={3},
       pages={449–\ndash 484},
}

\bib{Neves2010}{article}{
      author={Neves, A.},
       title={Singularities of {L}agrangian mean curvature flow: monotone
  case},
        date={2010},
        ISSN={1073-2780},
     journal={Math. Res. Lett.},
      volume={17},
      number={1},
       pages={109\ndash 126},
         url={https://doi.org/10.4310/MRL.2010.v17.n1.a9},
      review={\MR{2592731}},
}

\bib{Oh1994}{article}{
      author={Oh, Y.-G.},
       title={Mean curvature vector and symplectic topology of {Lagrangian}
  submanifolds in {Einstein-K\"ahler} manifolds},
        date={1994},
     journal={Mathematische Zeitschrift},
      volume={216},
      number={1},
       pages={471\ndash –482},
}

\bib{Perelman2002}{misc}{
      author={Perelman, G.},
       title={The entropy formula for the {Ricci} flow and its geometric
  applications},
        note={arXiv:math/0211159},
}

\bib{schlenk_chekanov_2010}{article}{
      author={Schlenk, F.},
      author={Chekanov, Y.},
       title={Notes on monotone {L}agrangian twist tori},
        date={2010},
     journal={Electronic Research Announcements in Mathematical Sciences},
      volume={17},
       pages={104–\ndash 121},
}

\bib{Smoczyk1996}{misc}{
      author={Smoczyk, K.},
       title={A canonical way to deform a {Lagrangian} submanifold},
         url={https://arxiv.org/pdf/dg-ga/9605005v2.pdf},
}

\bib{Smoczyk1999}{article}{
      author={Smoczyk, K.},
       title={{Harnack} inequality for the {Lagrangian} mean curvature flow},
        date={1999},
        ISSN={1432-0835},
     journal={Calculus of Variations and Partial Differential Equations},
      volume={8},
      number={3},
       pages={247\ndash 258},
         url={https://doi.org/10.1007/s005260050125},
}

\bib{sturm_1836}{article}{
      author={Sturm, C.},
       title={Sur une classe d'equations a differences partielles},
        date={1836},
     journal={J. Math. Pures Appl.},
      volume={1},
       pages={373–\ndash 444},
}

\bib{Thomas2002}{article}{
      author={Thomas, R.~P.},
      author={Yau, S.-T.},
       title={Special {Lagrangians}, stable bundles and mean curvature flow},
        date={2002},
     journal={Communications in Analysis and Geometry},
      volume={10},
      number={5},
       pages={1075–\ndash 1113},
}

\bib{vianna_2014}{article}{
      author={Vianna, R.},
       title={On exotic {L}agrangian tori in $\mathbb{CP}^2$},
        date={2014},
     journal={Geometry \& Topology},
      volume={18},
      number={4},
       pages={2419–\ndash 2476},
}

\bib{Vianna2016}{article}{
      author={Vianna, R.},
       title={Infinitely many exotic monotone {Lagrangian} tori in {$\bCP^2$}},
        date={2016},
     journal={Journal of Topology},
      volume={9},
      number={2},
       pages={535–\ndash 551},
}

\bib{Wang2001}{article}{
      author={Wang, M.-T.},
       title={Mean curvature flow of surfaces in {E}instein four-manifolds},
        date={2001},
     journal={J. Differential Geom.},
      volume={57},
      number={2},
       pages={301\ndash 338},
         url={https://doi.org/10.4310/jdg/1090348113},
}

\bib{Wood2019}{misc}{
      author={Wood, A.},
       title={Singularities of equivariant {L}agrangian mean curvature flow},
        date={2019},
         url={https://arxiv.org/abs/1910.06122},
        note={ArXiv preprint, arXiv:1910.06122},
}

\bib{wood_2020}{thesis}{
      author={Wood, A.},
       title={Singularities of {L}agrangian mean curvature flow},
        type={Ph.D. Thesis},
organization={University College London},
        date={2020},
}

\end{biblist}
\end{bibdiv}
	
\end{document}